\documentclass{aims}
\usepackage{amsmath,amsfonts,amssymb}
  \usepackage{paralist}
  \usepackage{graphics} 
  \usepackage{epsfig} 
\usepackage{graphicx}  \usepackage{epstopdf}
 \usepackage[colorlinks=true]{hyperref}
\hypersetup{urlcolor=blue, citecolor=red}

\def\currentvolume{X}
 \def\currentissue{X}
  \def\currentyear{200X}
   \def\currentmonth{XX}
    \def\ppages{X--XX}
     \def\DOI{10.3934/xx.xx.xx.xx}
     
\newtheorem{theorem}{Theorem}[section]

\newtheorem{lemma}[theorem]{Lemma}
\newtheorem{proposition}{Proposition}

\theoremstyle{definition}

\newtheorem{hypothesis}[theorem]{Hypothesis}
\newtheorem{remark}{Remark}

\def \Hm {\mathbb{H}}
\def \Imm {\mathbb{I}}
\def \Mm {\mathbb{M}}

\def \Rm {\mathbb{R}}
\def \Sm {\mathbb{S}}

\def \Vm {\mathbb{V}}

\def\C{\mathcal{C}}

\def\H{\mathcal{H}}

\def\L{\mathcal{L}}

\def\O{\mathcal{O}}

\newcommand{\dev}{\text{dev}}
\newcommand{\wtM}{ {\widetilde{M}}}
\newcommand{\dH}{ dH }

\newcommand{\where}{\quad\text{ where }}
\newcommand{\qandq}{\quad\text{ and }\quad}

\newcommand{\bfe}{ {\bf e}}

\newcommand{\bfh}{ {\bf h}}

\newcommand{\bfv}{ {\bf v}}

\newcommand{\bfI}{ {\bf I}}
\newcommand{\bfP}{ {\bf P}}

\newcommand{\tr}{ {\text{tr }}}

\def \hxi { \hat{\xi}}

\def \supp { {\mbox{supp}}}

\def \sp {\text{sp}}

\def \i {\boldsymbol\iota}

\title[Linearized internal functionals for anisotropic conductivities]
      {Linearized internal functionals for anisotropic conductivities}

      \author[Guillaume Bal Chenxi Guo and Fran\c{c}ois Monard]{}


 \email{gb2030@columbia.edu}
 \email{cg2597@columbia.edu}
 \email{fmonard@uw.edu}

\thanks{This work was supported in part by NSF grant DMS-1108608 and AFOSR Grant NSSEFF- FA9550-10-1-0194. FM is partially supported by NSF grant DMS-1025372.}

\begin{document}
\maketitle

\centerline{\scshape Guillaume Bal and Chenxi Guo}
\medskip
{\footnotesize
 \centerline{Department of Applied Physics and Applied Mathematics}
   \centerline{Columbia University}
   \centerline{New York, NY 10027, USA}
} 

\medskip

\centerline{\scshape Fran\c{c}ois Monard}
\medskip
{\footnotesize
 \centerline{ Department of Mathematics}
   \centerline{University of Washington}
   \centerline{Seattle WA, 98195, USA}
}

\bigskip

 \centerline{(Communicated by the associate editor name)}

\begin{abstract}
  This paper concerns the reconstruction of an anisotropic conductivity tensor in an elliptic second-order equation from knowledge of the so-called power density functionals. This problem finds applications in several coupled-physics medical imaging modalities such as ultrasound modulated electrical impedance tomography and impedance-acoustic tomography.
  
  We consider the linearization of the nonlinear hybrid inverse problem. We find sufficient conditions for the linearized problem, a system of partial differential equations, to be elliptic and for the system to be injective. Such conditions are found to hold for a lesser number of measurements than those required in recently established explicit reconstruction procedures for the nonlinear problem.
\end{abstract}

\section{Introduction}

In the context of hybrid medical imaging methods, a physical coupling between a high-contrast modality (e.g. Electrical Impedance Tomography, Optical Tomography) and a high-resolution modality (e.g. acoustic waves, Magnetic Resonance Imaging) is used in order to benefit from the advantages of both. Without this coupling, the high-contrast modality, usually modeled by an inverse problem involving the reconstruction of the constitutive parameter of an elliptic PDE from knowledge of boundary functionals, results in a mathematically severely ill-posed problem and suffers from poor resolution. The analysis of this coupling usually involves a two-step inversion procedure where the high-resolution modality provides internal functionals, from which we reconstruct the parameters of the elliptic equation, thus leading to improved resolution \cite{AS-IP-12,B-IO-12,Kuchment2012,Monard2012b,SU-IO-12}.

A problem that has received a lot of attention recently concerns the reconstruction of the conductivity tensor $\gamma$ in the elliptic equation
\begin{align}
    \nabla\cdot(\gamma\nabla u) = 0 \quad (X), \qquad u|_{\partial X} = g,
    \label{eq:conductivity}
\end{align}
from knowledge of internal power density measurements of the form $\nabla u\cdot\gamma\nabla v$, where $u$ and $v$ both solve \eqref{eq:conductivity} with possibly different boundary conditions. This problem is motivated by a coupling between electrical impedance imaging and ultrasound imaging and also finds applications in thermo-acoustic imaging.

Explicit reconstruction procedures for the above non-linear problem have been established in \cite{Capdeboscq2009,Bal2011a,Monard2011a,Monard2011,Monard2012b}, successively in the 2D, 3D, and $n$D isotropic case, and then in the 2D and $n$D anisotropic case. In these articles, the number of functionals may be quite large. The analyses in \cite{Monard2012b} were recently summarized and pushed further in \cite{Monard2012a}. If one decomposes $\gamma$ into the product of a scalar function $\tau = (\det\gamma)^\frac{1}{n}$ and a scaled anisotropic structure $\tilde\gamma$ such that $\det\tilde\gamma=1$, the latter reference establishes explicit reconstruction formulas for both quantities with Lipschitz stability for $\tau$ in $W^{1,\infty}$, and involving the loss of one derivative for $\tilde\gamma$.

In the isotropic case, several works study the above problem in the presence of a lesser number of functionals. The case of one functional is addressed in \cite{B-UMEIT-12}, whereas numerical simulations show good results with two functionals in dimension $n=2$ \cite{ABCTF-SIAP-08,GS-SIAP-09}. Theoretical and numerical analyses of the linearized inverse problem are considered in \cite{Kuchment2011a,Kuchment2011}. The stabilizing nature of  a class of internal functionals containing the power densities is demonstrated in \cite{Kuchment2011} using a micro-local analysis of the linearized inverse problem. The above inverse problem is recast as a system of nonlinear partial differential equations in \cite{B-Irvine-12} and its linearization analyzed by means of theories of elliptic systems of equations. It is shown in the latter reference that $n+1$ functionals, where $n$ is spatial dimension, is sufficient to reconstruct a scalar coefficient $\gamma$ with elliptic regularity, i.e., with no loss of derivatives, from power density measurements. This was confirmed by two-dimensional simulations in \cite{BNSS-JIIP-13}. All known explicit reconstruction procedures require knowledge of a larger number of internal functionals.



In the present work, we study the linearized version of this inverse problem in the anisotropic case, i.e. we write an expansion of the form $\gamma^\varepsilon = \gamma_0 + \varepsilon\gamma$ with $\gamma_0$ known and $\varepsilon\ll1$, and study the reconstructibility of $\gamma$ from linearized power densities (LPD). We first proceed by supporting the perturbation $\gamma$ away from the boundary $\partial X$ and analyze microlocally the symbol of the linearized functionals, and show that, as in \cite{Kuchment2011}, a large enough number of functionals allows us to construct a left-parametrix and set up a Fredholm inversion. The main difference between the isotropic and anisotropic settings is that the anisotropic part of the conductivity is reconstructed with a loss of one derivative. Such a loss of a derivative is optimal since our estimates are elliptic in nature. It is reminiscent of results obtained for a similar problem in \cite{BU-CPAM-12}.


Secondly, we show how the explicit inversion approach presented in \cite{Monard2012b,Monard2012a} carries through linearization, thus allowing for reconstruction of fully anisotropic tensors supported up to the boundary of $X$. In this case, we derive reconstruction formulas that require a smaller number of power densities than in the non-linear case, giving possible room for improvement in the non-linear inversion algorithms.

For additional information on hybrid inverse problems in other areas of (mostly medical) imaging, we refer the reader to, e.g., \cite{AS-IP-12,B-IO-12,S-SP-2011,SU-IO-12}.

\section{Statement of the main results}

Consider the conductivity equation \eqref{eq:conductivity}, where $X\subset\Rm^n$ is open, bounded and connected with $n\ge 2$, and where $\gamma^\varepsilon$ is a uniformly elliptic conductivity tensor over $X$.

We set boundary conditions $(g_1,\dots,g_m)$ and call $u_i^\varepsilon$ the unique solution to \eqref{eq:conductivity} with $u_i^\varepsilon|_{\partial X} = g_i$, $1\le i\le m$ and conductivity $\gamma^\varepsilon$. We consider the measurement functionals
\begin{align}
    H^\varepsilon_{ij}: \gamma^\varepsilon \mapsto H^\varepsilon_{ij}(\gamma^\varepsilon) = \nabla u^\varepsilon_i \cdot \gamma^\varepsilon\nabla u^\varepsilon_j (x), \quad 1\le i,j\le m, \quad x\in X.
    \label{eq:meas}
\end{align}

Considering an expansion of the form $\gamma^\varepsilon = \gamma_0 + \varepsilon \gamma$, where the background conductivity $\gamma_0$ is known, uniformly elliptic and $\varepsilon$ so small that the total $\gamma^\varepsilon$ remains uniformly elliptic, we first look for the Fr\'echet derivative of \eqref{eq:meas} with respect to $\gamma$ at $\gamma_0$. Expanding the solutions $u_i^\varepsilon$ accordingly as
\begin{align*}
    u_i^\varepsilon &= u_i + \varepsilon v_i + \O(\varepsilon^2),\quad 1\le i\le m,
\end{align*}
the PDE \eqref{eq:conductivity} at orders $\O(\varepsilon^0)$ and $\O(\varepsilon^1)$ gives rise to two relations
\begin{align}
    -\nabla\cdot(\gamma_0\nabla u_i) &= 0 \quad (X), \qquad u_i|_{\partial X} = g_i, \label{eq:conductivity0} \\
    -\nabla\cdot(\gamma_0\nabla v_i) &= \nabla\cdot(\gamma\nabla u_i) \quad (X), \qquad v_i|_{\partial X} = 0. \label{eq:conductivity1}
\end{align}
The measurements then look like
\begin{align}
    H^\varepsilon_{ij} = \nabla u_i\cdot\gamma_0 \nabla u_j + \varepsilon \left( \nabla u_i\cdot\gamma\nabla u_j +  \nabla u_i\cdot\gamma_0\nabla v_j + \nabla u_j\cdot\gamma_0\nabla v_i \right) + \O(\varepsilon^2).
\end{align}
Therefore, the component $\dH_{ij}$ of the Fr\'echet derivative of $H$ at $\gamma_0$ is
\begin{align}
    \dH_{ij} (\gamma) = \nabla u_i\cdot\gamma\nabla u_j +  \nabla u_i\cdot\gamma_0\nabla v_j + \nabla u_j\cdot\gamma_0\nabla v_i, \quad x\in X,
    \label{eq:LPD}
\end{align}
where the $v_i$'s are linear functions in $\gamma$ according to \eqref{eq:conductivity1}.

In both subsequent approaches, reconstruction formulas are established under the following two assumptions about the behavior of solutions related to the conductivity of reference $\gamma_0$. The first hypothesis deals with having a basis of gradients of solutions of \eqref{eq:conductivity0} over a certain subset $\Omega\subseteq X$.

\begin{hypothesis}\label{hyp:det} For an open set $\Omega\subseteq X$, there exist $(g_1,\dots,g_n)\in H^\frac{1}{2}(\partial X)^n$ such that the corresponding solutions $(u_1,\dots,u_n)$ of \eqref{eq:conductivity0} with boundary condition $u_i|_{\partial X} = g_i$ ($1\le i\le n$) satisfy
    \begin{align*}
	\inf_{x\in \Omega} \det (\nabla u_1,\dots,\nabla u_n) \ge c_0 >0.
    \end{align*}
\end{hypothesis}

Once Hypothesis \ref{hyp:det} is satisfied, any additional solution $u_{n+1}$ of \eqref{eq:conductivity0} gives rise to a $n\times n$ matrix
\begin{align}
    Z = [Z_1|\dots|Z_n], \where \quad Z_i:= \nabla \frac{\det (\nabla u_1,\dots,\overbrace{\nabla u_{n+1}}^i, \dots, \nabla u_n)}{\det (\nabla u_1,\dots,\nabla u_n)}.
    \label{eq:Zmat}
\end{align}
As seen in \cite{Monard2012b,Monard2012a}, such matrices can be computed from the power densities $\{\nabla u_i\cdot\gamma_0 \nabla u_j\}_{i,j=1}^{n+1}$ and help impose orthogonality conditions on the anisotropic part of $\gamma_0$. Once enough such conditions are obtained by considering enough additional solutions, then the anisotropy is reconstructed explicitly via a generalization of the usual cross-product defined in three dimensions. In the linearized setting, we find that {\em one} additional solution such that $Z$ has full rank is enough to reconstruct the linear perturbation $\gamma$. We thus formulate our second crucial assumption here:
\begin{hypothesis}\label{hyp:Z}
    Assume that Hypothesis \ref{hyp:det} holds over some fixed $\Omega\subseteq X$. There exists $g_{n+1}\in H^{\frac{1}{2}}(\partial X)$ such that the solution $u_{n+1}$ of \eqref{eq:conductivity0} with boundary condition $u_{n+1}|_{\partial X} = g_{n+1}$ has a full-rank matrix $Z$ (as defined in \eqref{eq:Zmat}) over $\Omega$.
\end{hypothesis}

\begin{remark}[Case $\gamma_0$ constant]\label{rem:const} In the case where $\gamma_0$ is constant, then it is straightforward to see that $g_i = x_i|_{\partial X}$ ($1\le i\le n$) fulfill Hypothesis \ref{hyp:det} over $X$. Moreover, if $Q = \{q_{ij}\}_{i,j=1}^n$ denotes an invertible constant matrix such that $Q:\gamma_0 = 0$, then the boundary condition $g_{n+1} := \frac{1}{2} q_{ij} x_i x_j |_{\partial X}$ fulfills Hypothesis \ref{hyp:Z}, since we have $Q = Z$.
\end{remark}

Throughout the paper, we use for (real-valued) square matrices $A$ and $B$ the contraction notation $A:B={\rm tr}\, AB^T=\sum_{i,j} A_{ij}B_{ij}$, with $B^T$ the transpose matrix of $A$.

\begin{remark}
    In the treatment of the non-linear case \cite{Bal2011a,Monard2011a,Monard2012b,Monard2012a}, it has been pointed out that Hypothesis \ref{hyp:det} may not be systematically satisfied globally in dimension $n\ge 3$. A more general hypothesis to consider would come from picking a larger family (of cardinality $>n$) of solutions whose gradients have maximal rank throughout $X$. While this additional technical point would not alter qualitatively the present reconstruction algorithms, it would add complexity in notation which the authors decided to avoid; see also \cite{BU-CPAM-12}.
\end{remark}

\subsection{Past work and heuristics for the linearization} \label{sec:heur}
In the reconstruction approach developped in \cite{Monard2011,Monard2012b,Monard2012a} for the non-linear problem, it was shown that not every part of the conductivity was reconstructed with the same stability. Namely, consider the decomposition of the tensor $\gamma'$ into the product of a scalar function $\tau = (\det\gamma')^\frac{1}{n}$ and a scaled anisotropic structure $\tilde\gamma'$ with $\det\tilde\gamma' = 1$. The following results were then established. Starting from $n$ solutions whose gradients form a basis of $\Rm^n$ over a subset $\Omega\subset X$, it was shown that under knowledge of a $W^{1,\infty}(X)$ anisotropic structure $\tilde\gamma'$, the scalar function $\log\det\gamma'$ was uniquely and Lipschitz-stably reconstructible in $W^{1,\infty}(\Omega)$ from $W^{1,\infty}$ power densities. Additionally, if one added a finite number of solutions $u_{n+1},\dots,u_{n+l}$ such that the family of matrices $Z_{(1)}, \dots, Z_{(l)}$ defined as in \eqref{eq:Zmat} imposed enough orthogonality constraints on $\tilde\gamma'$, then the latter was explicitely reconstructible over $\Omega$ from the mutual power densities of $(u_1,\dots,u_{n+l})$. The latter reconstruction was stable in $L^\infty$ for power densities in $W^{1,\infty}$ norm, thus it involved the loss of one derivative.

Passing to the linearized setting now (recall $\gamma^\varepsilon = \gamma_0 + \varepsilon\gamma$), and anticipating that one scalar quantity may be more stably reconstructible than the others, this quantity should be the linearized version of $\log\det\gamma^\varepsilon$. Standard calculations yield
\begin{align*}
    \log\det (\gamma_0 + \varepsilon\gamma) = \log\det \gamma_0 + \log\det (\Imm_n + \varepsilon \gamma_0^{-1}\gamma) = \log\det \gamma_0 + \varepsilon\tr (\gamma_0^{-1}\gamma ) + \O(\varepsilon^2),
\end{align*}
and thus the quantity that should be stably reconstructible is $\tr (\gamma_0^{-1} \gamma)$. The linearization of the product decomposition $(\tau,\tilde\gamma')$ above is now a spherical-deviatoric one of the form
\begin{align}
    \gamma = \frac{1}{n} \tr(\gamma_0^{-1}\gamma) \gamma_0 + \gamma_d, \quad \gamma_d := \gamma_0 (\gamma_0^{-1}\gamma)^\dev,
    \label{eq:gammadecomp}
\end{align}
where $^\dev$ is the linear projection onto the hyperplane of traceless matrices $A^\dev := A-\frac{\tr A}{n} \Imm_n$.

\subsection{Microlocal inversion}\label{ssec:statmicroloc}

The above inverse  problem in \eqref{eq:conductivity1}-\eqref{eq:LPD} may be seen as a system of partial differential equations for $(\gamma,\{v_j\})$. This is the point of view considered in \cite{B-Irvine-12}. However, $\{v_j\}$ may be calculated from \eqref{eq:conductivity1} and the expression plugged back into \eqref{eq:LPD}. This allows us to recast $dH$ as a linear operator for $\gamma$, which is smaller than the original linear system for $(\gamma,\{v_j\})$, but which is no longer differential and rather pseudo-differential. The objective in this section is to show, following earlier work in the isotropic case in \cite{Kuchment2011}, that such an operator is elliptic under appropriate conditions. 

We first fix $\Omega'\subset\subset X$ and assume that $\supp\gamma\subset\Omega'$, so that integrals of the form $\int_{\Rm^n} e^{\i x\cdot\xi} p(x,\xi):\hat\gamma(\xi)d\xi$ are well-defined, with $p(x,\xi)$ a matrix-valued symbol whose entries are polynomials in $\xi$ (see \cite[p.267]{Folland1995}) and where the hat denotes the Fourier Transform $\hat\gamma(\xi) = \int_{\Rm^n} e^{-\i x\cdot\xi} \gamma(x) \ dx$. We also assume that $\gamma_0\in\C^\infty(\Omega')$ and can be extended smoothly by $\gamma_0 = \Imm_n$ outside $\Omega'$. As pointed out in \cite{Kuchment2011}, in order to treat this problem microlocally, one must introduce cutoff versions of the $\dH_{ij}$ operators, which in turn extend to pseudo-differential operators ($\Psi$DO) on $\Rm^n$. Namely, if $\Omega''$ is a domain satisfying $\Omega'\subset\subset\Omega''\subset\subset X$ and $\chi_1$ is a smooth function supported in $X$ which is identically equal to $1$ on a neighborhood of $\overline{\Omega''}$, the operator $\gamma\mapsto \chi_1 \dH_{ij}(\chi_1 \gamma)$ can be made a $\Psi$DO upon considering $L_0 = -\nabla\cdot(\gamma_0 \nabla )$ as a second-order operator on $\Rm^n$ and using standard pseudo-differential parametrices to invert it \cite{GS-CUP-94}. We will therefore not distinguish the operators $\dH_{ij}$ from their pseudo-differential counterparts. The task of this section is then to determine conditions under which a given collection of such functions becomes an elliptic operator of $\gamma$ over $\Omega'$.

Using relations \eqref{eq:conductivity1} and \eqref{eq:LPD}, we aim at writing the operator $dH_{ij}$ in the following form
\begin{align}
    \dH_{ij}(x) &= (2\pi)^{-n} \int_{\Rm^n} \int_{\Rm^n} e^{\i\xi\cdot(x-y)} M_{ij}(x,\xi):\gamma(y)\ d\xi\ dy,
    \label{eq:dHijpsido}
\end{align}
with symbol $M_{ij}(x,\xi)$ (pseudo-differential terminology is recalled in Sec. \ref{ssec:prelim}). We first compute the main terms in the symbol expansion of $\dH_{ij}$ (call this expansion $M_{ij} = M_{ij}|_0 + M_{ij}|_{-1} + \O (|\xi|^{-2})$ with $M_{ij}|_p$ homogeneous of degree $p$ in $\xi$). From these expressions, we then directly deduce microlocal properties on the corresponding operators.

The first lemma shows that the principal symbols $M_{ij}|_0$ can never fully invert for $\gamma$, no matter how many solutions $u_i$ we pick. When Hypothesis \ref{hyp:det} is satisfied, then the characteristic directions of the principal symbols $\{M_{ij}(x,\xi)\}_{1\le i,j\le n}$ reduce to a $n-1$-dimensional subspace of $S_n(\Rm)$. Here and below, we recall that the colon ``$:$'' denotes the inner product $A:B = \tr (AB^T)$ for $(A,B)\in S_n(\Rm)$ and $\odot$ denotes the symmetric outer product $U\odot V = \frac{1}{2} (U\otimes V + V\otimes U)$ for $U,V\in \Rm^n$.
\begin{lemma} \label{lem:xieta}
    \begin{itemize}
	\item[(i)] For any $i,j$ and $x\in X$, the symbol $M_{ij}|_0$ satisfies
	    \begin{align}
		M_{ij}|_0 : (\gamma_0\xi\odot\eta) = 0, \quad \text{for all } \,\eta\,\in \Sm^{n-1} \text{ satisfying } \quad \eta\cdot\xi = 0.
		\label{eq:xieta}
	    \end{align}
	\item[(ii)] Suppose that Hypothesis \ref{hyp:det} holds over some $\Omega\subseteq X$. Then for any $x\in \Omega$, if $P\in S_n(\Rm)$ is such that
	    \begin{align}
		M_{ij}|_0 : P=0, \quad 1\le i\le j\le n,
		\label{eq:Portho}
	    \end{align}
	    then $P$ is of the form $P = \gamma_0\xi\odot\eta$ for some vector $\eta$ satisfying $\eta\cdot\xi = 0$.
    \end{itemize}
\end{lemma}
Since an arbitrary number of zero-th order symbols can never be elliptic with respect to $\gamma$, we then consider the next term in the symbol expansion of $\dH_{ij}$. We must also add one solution $u_{n+1}$ to the initial collection, exhibiting appropriate behavior, i.e. satisfying Hypothesis \ref{hyp:Z}. The collection of functionals we consider below is thus of the form 
\begin{align}
    \dH := \{\dH_{ij}\, |\, 1\le i\le n,\ i\le j\le n+1 \},
    \label{eq:dHcollection}
\end{align}
and emanates from $n+1$ solutions $(u_1,\dots,u_{n+1})$ of \eqref{eq:conductivity0} satisfying Hypotheses \ref{hyp:det} and \ref{hyp:Z}. 

In order to formulate the result, we assume to construct a family of unit vector fields 
\begin{align*}
    \hxi_0(x,\xi) := \widehat{A_0(x)\xi}\, , \, \hxi_1(x,\xi)\,,\, \dots\,,\, \hxi_{n-1}(x,\xi),
\end{align*}
homogeneous of degree zero in $\xi$, smooth in $x$ and everywhere orthonormal. We then define the family of scalar elliptic zeroth-order $\Psi$DO
\begin{align}
    \begin{split}
	T: \gamma &\mapsto T\gamma = \{T_{pq}\gamma\}_{0\le p\le q\le n-1}, \quad \where \\
	T_{pq}\gamma (x) &:= (2\pi)^{-n} \int_{\Rm^n} e^{\i\xi\cdot x} A_0^{-1}\hxi_p \odot \hxi_q A_0^{-1} : \hat\gamma(\xi)\ d\xi, \quad 1\le p\le q\le n,
    \end{split}
    \label{eq:T}
\end{align}
which can be thought of as a microlocal change of basis after which the operator $dH(\gamma)$ becomes both diagonal and elliptic. Indeed, we verify (see section \ref{ssec:proofs}) that for any $k\geq1$ and $\gamma$ sufficiently regular, we have
\begin{equation}\label{eq:changebasis}
   \|\gamma\|_{H^k(\Omega')} \leq C \|T\gamma\|_{H^k(\Omega')} + C_2 \|\gamma\|_{L^2(\Omega')} \leq C_3 \|\gamma\|_{H^k(\Omega')}.
\end{equation}
 The above estimates come from standard result on pseudo-differential operators \cite{GS-CUP-94}. The presence of the constant $C_2$ indicates that $T$ can be inverted microlocally, but may not injective.
 
   Composing the measurements $\dH_{ij}$ with appropriate scalar $\Psi$DO of order 0 and 1, we are then able to recover each component of the operator \eqref{eq:T}. The well-chosen ``parametrices'' are made possible by the fact that the collection of symbols $M_{ij}|_0 + M_{ij}|_{-1}$ becomes elliptic over $\Omega'$ when Hypotheses \ref{hyp:det} and \ref{hyp:Z} are satisfied. Rather than using the full collection of measurements $\dH$ \eqref{eq:dHcollection}, we will consider the smaller collection $\{\dH_{ij}\}_{1\le i,j\le n}$ augmented with the $n$ measurement operators 
\begin{align}
    L_i(\gamma) = \sum_{j=1}^n \mu_j\ \dH_{ij} (\gamma) + \mu\ \dH_{i,n+1}(\gamma), \quad 1\le i\le n,
    \label{eq:Li}
\end{align}
where $(\mu_1,\dots,\mu_n,\mu)(x)$, known from the measurements $\{H_{ij}\}_{i,j=1}^{n+1}$, are the coefficients in the relation of linear dependence
\begin{align*}
    \mu_1 \nabla u_1 + \dots + \mu_n \nabla u_n + \mu\nabla u_{n+1} = 0.
\end{align*}
We also define the operator $L_0^{\frac{1}{2}}\in \Psi^1$ with principal symbol $-\i \|A_0 \xi\|$. 
Our conclusions may be formulated as follows:
\begin{proposition} \label{prop:microloc}
    Let the measurements $\dH$ defined in \eqref{eq:dHcollection} satisfy Hypotheses \ref{hyp:det} and \ref{hyp:Z}. 
    \begin{itemize}
	\item[(i)] For $(\alpha,\beta) = (0,0)$ and $1\le \alpha\le \beta\le n-1$, there exist $\{Q_{\alpha\beta ij}\}_{1\le i\le j\le n}\in \Psi^0$ such that 
	    \begin{align}
		\sum_{1\le i,j\le n} Q_{\alpha\beta ij} \circ \dH_{ij} = T_{\alpha\beta} \quad \mod \Psi^{-1}. 
		\label{eq:param1}
	    \end{align}
	\item[(ii)] For any $1\le \alpha\le n-1$, there exist $\{B_{\alpha i}\}_{1\le i\le n}\in \Psi^0$ such that the following relation holds
	    \begin{align}
		L_0^{\frac{1}{2}} \circ B_{\alpha i} \circ L_i - R_\alpha \circ R = T_{0\alpha} \quad \mod \Psi^{-1}, 
		\label{eq:param2}
	    \end{align}
	    where the remainder $R_\alpha \circ R$ can be expressed as a zeroth-order linear combination of the components $T_{00}$ and $\{T_{pq}\}_{1\le p\le q \le n-1}$  reconstructed in (i). 
    \end{itemize}

\end{proposition}

The presence of the $L_0^{\frac{1}{2}}$ term in part {\em (ii)} of Prop. \ref{prop:microloc} accounts for the loss of one derivative in the inversion process. From Prop. \ref{prop:microloc}, we can then obtain stability estimates of the form 
\begin{align}
    \begin{split}
	\|T_{00}\gamma\|_{H^{k+1}(\Omega')} &+ \!\!\!\!\!  \sum_{1\le p\le q\le n-1} \!\!\!\!\!\!\!\!   \|T_{pq}\gamma\|_{H^{k+1}(\Omega')} +  \!\! \!\!\! \sum_{1\le p\le n-1}  \!\! \!\!\! \|T_{0p}\gamma\|_{H^k(\Omega')} \\ 
	&\le C \|dH\|_{H^{k+1}(\Omega')} + C_2 \|\gamma\|_{L^2(\Omega')}.	
    \end{split}     
    \label{eq:microlocstab}
\end{align}
The above stability estimate holds for $k=0$ using the results of Proposition \ref{prop:microloc} and in fact for any $k\geq0$ updating by standard methods (not detailed here \cite{GS-CUP-94}) the parametrices in \eqref{eq:param1} and \eqref{eq:param2} to inversions modulo operators in $\Psi^{-k}$ (i.e., classical $\Psi$DO of order $-k$ \cite{GS-CUP-94}) provided that the coefficients $(\gamma_0,\{u_j\})$ are sufficiently smooth. The presence of the constant $C_2$ indicates that the reconstruction of $\gamma$ may be performed up to the existence of a finite dimensional kernel as an application of the Fredholm theory as in \cite{Kuchment2011}.

Equation \eqref{eq:microlocstab} means that some components of $\gamma$ are reconstructed with a loss of one derivative while other components are reconstructed with no loss. The latter components are those that can be spanned by the components $T_{00}\gamma$ and $\{T_{\alpha\beta}\gamma\}_{1\le\alpha,\beta\le n-1}$. Some algebra shows that the only such linear combination is $\sum_{i=0}^{n-1} T_{ii} \gamma$, which, using the fact that $\sum_{i=0}^{n-1} \hxi_i\otimes\hxi_i = \Imm_n$, can be computed as
\begin{align*}
    \sum_{i=0}^{n-1} T_{ii} \gamma &= (2\pi)^{-n} \int_{\Rm^n} e^{\i x\cdot\xi} A_0^{-1} \Imm_n A_0^{-1} : \hat\gamma(\xi)\ d\xi \\
    &= \gamma_0^{-1}: \left( (2\pi)^{-n} \int_{\Rm^n} e^{\i x\cdot\xi} \hat\gamma(\xi)\ d\xi \right) = \tr (\gamma_0^{-1}\gamma),
\end{align*}
confirming the heuristics of Sec. \ref{sec:heur}. It can be shown that all other components of $\gamma$ (i.e. any part of $\gamma_d$ in \eqref{eq:gammadecomp}) are, to some extent, spanned by the components $T_{0\alpha}\gamma$, and as such cannot be reconstructed with better stability than the loss of one derivative in light of \eqref{eq:microlocstab}. Combining the above results with \eqref{eq:changebasis}, we arrive at the main stability result of the paper:
\begin{align}\label{eq:stabfinal}
  \| \tr(\gamma_0^{-1}\gamma) \|_{H^{k}(\Omega')} + \|\gamma_d\|_{H^{k-1}(\Omega')} \le C \|\dH\|_{H^k(\Omega')} + C_2 \|\gamma\| _{L^2(\Omega')}.
\end{align}
Such an estimate holds for any $k\geq1$.

The above estimate holds with $C_2=0$ when $\gamma\mapsto dH(\gamma)$ is an injective (linear) operator. Injectivity cannot be verified by microlocal arguments since all inversions are performed up to smoothing operators; see \cite{B-Irvine-12} in the isotropic setting. In the next section, we obtain an injectivity result, which allows us to set $C_2=0$ in the above expression. However, the above stability estimate \eqref{eq:stabfinal} is essentially optimal. An optimal estimate, which follows from the above and the equations for $(\gamma,\{v_j\})$ is the following:
\begin{align*}
  \|M_{|0}\gamma \| _{H^{k}(\Omega')} + \|\gamma\|_{H^{k-1}(\Omega')} &\leq  C \|\dH\|_{H^k(\Omega')} + C_2 \|\gamma\| _{L^2(\Omega')} \\
  &\leq C'(\|M_{|0}\gamma \| _{H^{k}(\Omega')} + \|\gamma\|_{H^{k-1}(\Omega')} ).
\end{align*}
The left-hand-side inequality is a direct consequence of \eqref{eq:stabfinal} and the expression of $dH$. The right-hand side is a direct consequence of the expression of $dH$. The above estimate is clearly optimal. The operator $M_{|0}$ is of order $0$. If it were elliptic, then $\gamma$ would be reconstructed with no loss of derivative. However, $M_{|0}$ is not elliptic and the loss of ellipticity is precisely accounted for by the results in Lemma \ref{lem:xieta}. As we discussed above, it turns out that the only spatial coefficient controlled by $M_{|0}\gamma$ is $\tr(\gamma_0^{-1}\gamma)$, and hence \eqref{eq:stabfinal}.

\subsection{Explicit inversion:}

Now, allowing $\gamma$ to be supported up to the boundary, we present a variation of the non-linear resolution technique used in \cite{Monard2012b,Monard2012a}. First considering $n$ solutions generated by boundary conditions fulfilling Hypothesis \ref{hyp:det}, we establish an expression for $\gamma$ in terms of the remaining unknowns $(v_1,\dots,v_n)$:
\begin{align}
    \gamma = \gamma_0 ([\nabla U] H^{-1} \dH H^{-1} [\nabla U]^T - [\nabla V] H^{-1} [\nabla U]^T - [\nabla U] H^{-1} [\nabla V]^T ) \gamma_0,
    \label{eq:gammaelim}
\end{align}
where $[\nabla U]$ and $[\nabla V]$ denote $n\times n$ matrices whose $j$-th columns are $\nabla u_j$ and $\nabla v_j$, respectively, and where $H = \{H_{ij}\}_{i,j=1}^n$ and $\dH = \{\dH_{ij}\}_{i,j=1}^n$. In particular we find from \eqref{eq:gammaelim} the relation
\begin{align}
    \tr (\gamma_0^{-1} \gamma) = \tr (H^{-1} \dH) - 2\tr M, \quad M := ([\nabla V][\nabla U]^{-1})^T.
    \label{eq:trgamma}
\end{align}
Plugging \eqref{eq:gammaelim} back into the second equation in \eqref{eq:conductivity} for $1\le i\le n$, one can deduce a gradient equation for the quantity $\tr (\gamma_0^{-1}\gamma)$ which in turn allows to reconstruct $\tr (\gamma_0^{-1} \gamma)$ in a Lipschitz-stable manner with respect to the LPD $\{\dH_{ij}\}_{i,j=1}^n$ (i.e. without loss of derivative).

Now turning to the full reconstruction of $\gamma$, we consider an additional solution $u_{n+1}$ generated by a boundary condition fulfilling Hyp. \ref{hyp:Z}. The following proposition then establishes how to reconstruct $(v_1,\dots,v_n)$ from $\dH$:
\begin{proposition}\label{prop:scesv}
    Assume that $(g_1,\dots,g_{n+1})$ fulfill Hypotheses \ref{hyp:det} and \ref{hyp:Z} over $X$ and consider the linearized power densities $\dH = \{\dH_{ij}:\ 1\le i\le j\le n+1,\ i\ne n+1\}$. Then the solutions $(v_1,\dots,v_n)$ satisfy a strongly coupled elliptic system of the form
    \begin{align}
	-\nabla\cdot(\gamma_0 \nabla v_i) + W_{ij}\cdot\nabla v_j = f_i (\dH, \nabla (\dH)) \quad (X),\quad v_i|_{\partial X} = 0, \quad 1\le i\le n,
	\label{eq:scesv}
    \end{align}
    where the vector fields $W_{ij}$ are known and only depend on the behavior of $\gamma_0$, $Z$ and $u_1,\dots,u_n$, and where the functionals $f_i$ are linear in the data $\dH_{ij}$.
\end{proposition}

When the vector fields $W_{ij}$ are bounded, system \eqref{eq:scesv} satisfies a Fredholm alternative from which we deduce that if \eqref{eq:scesv} with a trivial right-hand side admits no non-trivial solution, then $(v_1,\dots,v_n)$ is uniquely reconstructed from \eqref{eq:scesv}. We can then reconstruct $\gamma$ from \eqref{eq:gammaelim}.

\begin{remark}[Case $\gamma_0$ constant]
    In the case where $\gamma_0$ is constant, choosing solutions as in Remark \ref{rem:const}, one arrives at a system of the form \eqref{eq:scesv} where $W_{ij} = 0$ if $i\ne j$, so that the system is decoupled and clearly injective.
\end{remark}

The conclusive theorem for the explicit inversion is thus given by

\begin{theorem}\label{thm:explicit}
    Assume that $(g_1,\dots,g_{n+1})$ fulfill Hypotheses \ref{hyp:det} and \ref{hyp:Z} over $X$ and consider the linearized power densities $\dH = \{\dH_{ij}:\ 1\le i\le j\le n+1,\ i\ne n+1\}$. Assume further that the system \eqref{eq:scesv} with trivial right-hand sides has no non-trivial solution. Then $\gamma$ is uniquely determined by $\dH$ and we have the following stability estimate
    \begin{align}
      \| \tr  (\gamma_0^{-1}\gamma)\|_{H^1(X)} + \|\gamma \|_{L^2(X)} \le C \|\dH \|_{H^1(X)}.
	\label{eq:stabgammaexplicit}
    \end{align}
\end{theorem}

\subsection{Outline}

We cover the microlocal inversion in Sec. \ref{sec:microloc}. Linear algebraic and pseudo-differential preliminaries are given in Sec. \ref{ssec:prelim}. The leading-order symbols of order $0$ of the LPD functionals are computed in Sec. \ref{ssec:symbol0} and a proof of Lemma \ref{lem:xieta} is given. The symbols of order $-1$ are then computed in \ref{ssec:symbolm1} and the proof Proposition \ref{prop:microloc} is given in Sec. \ref{ssec:proofs}. 
We then treat the explicit inversion in Sec. \ref{sec:explicit}. Starting with some preliminaries in Sec. \ref{ssec:prelim2}, we derive some crucial relations in Sections \ref{ssec:der1} and \ref{ssec:der2}, before proving Proposition \ref{prop:scesv} and Theorem \ref{thm:explicit} in Sec. \ref{ssec:proofs2}.

\section{Microlocal inversion} \label{sec:microloc}

\subsection{Preliminaries} \label{ssec:prelim}

\paragraph{Linear algebra.} In the following, we consider the $n\times n$ matrices $M_n(\Rm)$ with the inner product structure
\begin{align}
    A:B = \tr (AB^T) = \sum_{i,j=1}^n A_{ij}B_{ij},
    \label{eq:innprod}
\end{align}
for which $M_n(\Rm)$ admits the orthogonal decomposition $A_n(\Rm)\oplus S_n(\Rm)$. For two vectors $U= (u_1,\dots,u_n)^T$ and $V=(v_1,\dots,v_n)^T$ in $\Rm^n$ we denote by $U\otimes V$ the matrix with entries $\{u_i v_j\}_{i,j=1}^n$, and we also define the symmetrized outer product
\begin{align}
    U\odot V:= \frac{1}{2} (U\otimes V + V\otimes U).
    \label{eq:outerprod}
\end{align}
With $\cdot$ denoting the standard dotproduct on $\Rm^n$, we have the following identities
\begin{align}
    2 U\odot V: X\odot Y &= (U\cdot X)(V\cdot Y) + (U\cdot Y)(V\cdot X), \qquad U,V,X,Y \in \Rm^n, \label{eq:id} \\
    U\cdot MU &= M:U\otimes U = M:U\odot U, \qquad U\in \Rm^n, M \in M_n(\Rm). 
\end{align}

\paragraph{Pseudo-differential calculus.} Recall that we denote the set of {\em symbols} of order $m$ on $X$ by $S^m(X)$, which is the space of functions $p\in C^{\infty}(X\times\Rm^n)$ such that for all multi-indices $\alpha$ and $\beta$ and every compact set $K\subset X$ there is a constant $C_{\alpha,\beta,K}$ such that
\begin{align*}
    \sup_{x\in K}|D_x^\beta D_\xi^\alpha p(x,\xi)|\leq C_{\alpha,\beta,K}(1+|\xi|)^{m-|\alpha|}
\end{align*}
We denote the operator $p(x,D)$ as
\begin{align*}
  p(x,D)\gamma(x)= (2\pi)^{-n}\int_{\Rm^n} e^{\i x\cdot\xi}p(x,\xi)\hat \gamma(\xi)d\xi
\end{align*}
and the set of {\em pseudo-differential operators} ($\Psi$DO) of order $m$ on $X$ by $\Psi^m(X)$, where
\begin{align*}
    \Psi^m(X)=\{p(x,D):p\in S^m(X)\}.
\end{align*}
Suppose $\{m_j\}^\infty_0$ is strictly decreasing and $\lim m_j=-\infty$, and suppose $p_j\in S^{m_k}(X)$ for each $j$. We denote an {\em asymptotic expansion} of the symbol $p\in S^{m_0}(X)$ as $p\thicksim\sum_0^{\infty}p_j$
if
\begin{align*}
    p-\sum_{j<k}p_j\in S^{m_k}(X), \quad \text{for all } k>0.
\end{align*}
Given two $\Psi$DO $P$ and $Q$ with respective symbols $\sigma_P$ and $\sigma_Q$ and orders $d_P$ and $d_Q$, we will make repetitive use of the symbol expansion of the product operator $QP\equiv Q\circ P$ (see \cite[Theorem (8.37)]{Folland1995} for instance)
\begin{align}
    \sigma_{QP}(x,\xi) \sim \sigma_Q \sigma_P + \frac{1}{\i} \nabla_\xi \sigma_Q \cdot \nabla_x \sigma_{P} + \O(|\xi|^{d_Q+d_P-2}),
    \label{eq:prodpsidos}
\end{align}
where $\O(|\xi|^{\alpha})$ denotes a symbol of order at most $\alpha$. As we will need to compute products of three $\Psi$DO $R$, $P$ and $Q$, we write the following formula for later use, obtained by iteration of \eqref{eq:prodpsidos}
\begin{align}
    \begin{split}
	\sigma_{RQP} = \sigma_R \sigma_Q \sigma_P &+ \frac{1}{\i} (\sigma_R \nabla_\xi \sigma_Q\cdot\nabla_x \sigma_P + \sigma_Q \nabla_\xi \sigma_R \cdot\nabla_x \sigma_P + \sigma_P \nabla_\xi \sigma_R\cdot\nabla_x \sigma_Q) \\
	&\qquad \qquad  + \O (|\xi|^{d_R+d_Q+d_P-2}).	
    \end{split}    
    \label{eq:prodpsidos2}
\end{align}
In the next derivations, some operators have matrix-valued principal symbols. However we will only compose them with operators with scalar symbols, so that the above calculus remains valid. 

\subsection{Symbol calculus for the LPD, properties of $M_{ij}|_0$ and proof of Lemma \ref{lem:xieta}} \label{ssec:symbol0}
Writing $v_i(x) = (2\pi)^{-n} \int_{\Rm^n} e^{\i x\cdot\xi} \hat v_i(\xi)\ d\xi$ and $\gamma(x) = (2\pi)^{-n} \int_{\Rm^n} e^{\i x\cdot\xi} \hat \gamma(\xi)\ d\xi$ (understood in the componentwise sense), we have
\begin{align*}
    L_0 v_i := -\nabla\cdot(\gamma_0\nabla v_i) &= (2\pi)^{-n} \int_{\Rm^n} e^{\i x\cdot\xi} \left( \xi\cdot\gamma_0\xi - \i(\nabla\cdot\gamma_0)\cdot\xi \right) \hat v_i(\xi)\ d\xi, \\
    P_i \gamma \equiv \nabla\cdot(\gamma\nabla u_i) &= (2\pi)^{-n} \int_{\Rm^n} e^{\i x\cdot\xi} \left( \i\xi\odot\nabla u_i + \nabla^2 u_i \right):\hat\gamma(\xi)\ d\xi.
\end{align*}
Thus equation \eqref{eq:conductivity1} reads $L_0 v_i = P_i \gamma$, where the operators $L_0 := -\nabla\cdot(\gamma_0\nabla)$ and $P_i$ have respective symbols
\begin{align}
    \sigma_{L_0} &= l_2 + l_1, \qquad l_2 := \xi\cdot\gamma_0\xi \in S^{2} \qandq l_1:= - \i(\nabla\cdot\gamma_0)\cdot\xi \in S^1, \label{eq:L0} \\
    \sigma_{P_i} &= p_{i,1} + p_{i,0}, \qquad p_{i,1} := \i\xi\odot \nabla u_i \in (S^1)^{n\times n} \qandq p_{i,0} := \nabla^2 u_i \in (S^0)^{n\times n}. \label{eq:Pi}
\end{align}
For $Y$ a smooth vector field, we will also need in the sequel to express the operator $Y\cdot\nabla$ as $\Psi$DO, the symbol of which is denoted $\sigma_{Y\cdot\nabla} = \sigma_{Y\cdot\nabla}|_1 := \i\xi\cdot Y$.

We now write $\dH_{ij}$ as a $\Psi$DO of $\gamma$ with symbol $M_{ij}$ as in \eqref{eq:dHijpsido}. $\dH_{ij}$ belongs to $\Psi^{0}(X)$ and we will compute in this paper the first two terms in the expansion of $M_{ij}$ (call them $M_{ij}|_0$ and $M_{ij}|_{-1}$), which in turn relies on constructing parametrices of $L_0$ of increasing order and doing some computations on symbols of products of $\Psi$DO based on formula \eqref{eq:prodpsidos}. 
If $Q$ is a parametrix of $L_0$ modulo $\Psi^{-m}$, i.e. $K\equiv QL_0 - Id \in \Psi^{-m}$, then straightforward computations based on the relation $L_0 v_i = P_i \gamma$ yield the following relation
\begin{align}
    dH_{ij} (\gamma) &= \gamma:\nabla u_i \odot \nabla u_j + (\gamma_0 \nabla u_i\cdot\nabla) \circ Q\circ P_j \gamma + (\gamma_0 \nabla u_j\cdot\nabla) \circ Q\circ P_i \gamma + K_{ij}\gamma, \label{eq:dHijparam}\\
    \!\!\!\!\where &\quad K_{ij} := (\gamma_0\nabla u_i\cdot\nabla) \circ KL_0^{-1}P_j + (\gamma_0\nabla u_j\cdot\nabla) \circ KL_0^{-1}P_i.	\label{eq:Kp}        
\end{align}
For any $i$, $L_0^{-1}P_i$ denotes the operator $\gamma\mapsto v_i$ where $v_i$ solves \eqref{eq:conductivity1}, and standard elliptic theory allows to claim that $L_0^{-1}P_i$ smoothes by one derivative so that the error operator $K_{ij}$ defined in \eqref{eq:Kp} smoothes by $m$ derivatives. In particular, upon computing a parametrix $Q$ of $L_0$ modulo $\Psi^{-m}$, the first three terms in \eqref{eq:dHijparam} are enough to construct the principal part of the symbol $M_{ij}$ modulo $\Psi^{-m}$. 

\paragraph{Computation of $M_{ij}|_0$.} In light of the last remark, we first compute a parametrix $Q$ of $L_0$ modulo $\Psi^{-1}$, that is, since $L_0\in \Psi^{2}$, we look for a principal symbol of the form $\sigma_Q = q_{-2} + \O(|\xi|^{-3})$. Clearly, we easily obtain $q_{-2} = l_2^{-1} = (\xi\cdot\gamma_0\xi)^{-1}$. In this case, the principal symbol of $\dH_{ij}$ at order zero is given by, according to \eqref{eq:dHijparam} and \eqref{eq:prodpsidos},
\begin{align*}
    M_{ij}|_0 &= \nabla u_i\odot\nabla u_j + (\sigma_{\gamma_0\nabla u_i\cdot\nabla}|_1)\ q_{-2}\ p_{j,1} + (\sigma_{\gamma_0\nabla u_j\cdot\nabla}|_1)\ q_{-2}\ p_{i,1} \\
    &= \nabla u_i\odot\nabla u_j - \frac{1}{\xi\cdot\gamma_0\ \xi} \left( (\gamma_0\nabla u_i\cdot \xi) (\xi\odot\nabla u_j) + (\gamma_0\nabla u_j\cdot \xi) (\xi\odot\nabla u_i) \right).
\end{align*}
$M_{ij}|_0$ admits a somewhat more symmetric expression if pre- and post-multiplied by $A_0$, the unique positive squareroot of $\gamma_0$, so that we may write,
\begin{align}
    M_{ij}|_0 (x,\xi) = A_0^{-1} \left( V_i\odot V_j - (\hxi_0 \cdot V_i) \hxi_0\odot V_j - (\hxi_0\cdot V_j) \hxi_0\odot V_i  \right) A_0^{-1},
    \label{eq:Mij0_2}
\end{align}
where we have defined $\xi_0 := A_0\xi$ and $\hat{x} := |x|^{-1} x$ for any $x\in \Rm^n-\{0\}$ as well as $V_i:= A_0\nabla u_i$. This last expression motivates the proof of Lemma \ref{lem:xieta}.
\begin{proof}[Proof of Lemma \ref{lem:xieta}] {\bf Proof of {\em (i)}:} Let $\eta$ such that $\eta \cdot \xi = 0$, and denote $\eta' := A_0 ^{-1}\eta$ so that $\eta'\cdot\xi_0 = 0$. Then using identity \eqref{eq:id} and \eqref{eq:Mij0_2}, we get
    \begin{align*}
	2 \|\xi_0\|^{-1} M_{ij}|_0(x,\xi) &: \gamma_0 \xi \odot\eta \dots \\
	&= 2 \left[ V_i\odot V_j - (\hxi_0\cdot V_i) \hxi_0\odot V_j - (\hxi_0\cdot V_j) \hxi_0\odot V_i  \right]:\hxi_0\odot\eta' \\
	&= (V_i\cdot \hxi_0) (V_j\cdot \eta') + (V_i\cdot\eta')(V_j\cdot \hxi_0)  \\
	&\quad - (\hxi_0\cdot V_i) (\hxi_0\cdot\hxi_0) (V_j\cdot \eta') - (\hxi_0\cdot V_i)(\hxi_0\cdot\eta')(V_j\cdot\hxi_0)  \\
	&\quad - (\hxi_0\cdot V_j) (\hxi_0\cdot \hxi_0) (V_i\cdot\eta') - (\hxi_0\cdot V_j) (\hxi_0\cdot\eta') (\hxi_0\cdot V_i)\\
	&=0,
    \end{align*}
    where we have used $\hxi_0\cdot\hxi_0=1$ and $\hxi_0\cdot\eta' = 0$, thus {\em (i)} holds. \smallskip

    {\bf Proof of {\em (ii)}:} Recall that 
    \begin{align*}
	M_{ij}|_0 : P = \left[ V_i\odot V_j - (\hxi_0\cdot V_i) \hxi_0\odot V_j - (\hxi_0\cdot V_j) \hxi_0\odot V_i \right]: A_0^{-1} P A_0^{-1}. 
    \end{align*}
    We write $S_n(\Rm)$ as the direct orthogonal sum of three spaces:
    \begin{align}
	S_n(\Rm) = \left( \Rm\ \hxi_0\otimes \hxi_0 \right) \oplus \left(\{\hxi_0\}^\perp \odot \{\xi_0\}^\perp \right) \oplus \left( \hxi_0 \odot \{\hxi_0\}^\perp \right) ,
	\label{eq:Sndecomp}
    \end{align} 
    with respective dimensions $1$, $n(n-1)/2$ and $n-1$. Decomposing $A_0^{-1} P A_0^{-1}$ uniquely into this sum, we write $A_0^{-1} P A_0^{-1} = P_1 + P_2 + P_3$. Direct calculations then show that 
    \begin{align*}
	M_{ij}|_0 : A_0^{-1} P A_0^{-1} = V_i\odot V_j:(- P_1 + P_2), \quad 1\le i\le j\le n.
    \end{align*}
    Since $\{V_i\}_{i=1}^n$ is a basis of $\Rm^n$, $\{V_i\odot V_j\}_{1\le i\le j\le n}$ is a basis of $S_n(\Rm)$ and thus \eqref{eq:Portho} implies that
    \begin{align*}
	-P_1 + P_2 = 0, \qquad \text{i.e.} \qquad P_1 = P_2 = 0.
    \end{align*}
    Therefore $P =  A_0 P_3 A_0$ with $P_3 = \hxi_0 \odot \eta'$ for some $\eta'\cdot\hxi_0 = 0$, so $P = \gamma_0 \xi \odot \eta$ with $\eta$ proportional to $A_0 \eta'$, i.e. such that $\eta\cdot\xi = 0$, thus the proof is complete.  
\end{proof}

In other words, {\bf all} symbols of order zero $M_{ij}|_0 (x,\xi)$ are orthogonal to the $(x,\xi)$-dependent $n-1$-dimensional subspace of symmetric matrices $\gamma_0 \xi \odot \{\xi\}^\perp$. One must thus compute the next term in the symbol exampansion of the operators $\dH_{ij}$, i.e. $M_{ij}|_{-1}$. We will then show that enough symbols of the form $M_{ij}|_0 + M_{ij}|_{-1}$ will suffice to span the entire space $S_n(\Rm)$ for every $x\in \Omega'$ and $\xi\in \Sm^1$, so that the corresponding family of operators is elliptic as a function of $\gamma$.

\subsection{Computation of $M_{ij}|_{-1}$} \label{ssec:symbolm1}

As the previous section explained, the principal symbols $M_{ij}|_0$ can never span $S_n(\Rm)$. Therefore, we compute the next term $M_{ij}|_{-1}$ in their symbol expansion. We must first construct a parametrix $Q$ of $L_0$ modulo $\Psi^{-2}$, i.e. of the form 
\begin{align}
    \sigma_Q = q_{-2} + q_{-3} + \O(|\xi|^{-4}),\quad q_i \in S^{i}.
    \label{eq:paramm2}
\end{align}

\begin{lemma}\label{lem:qm2m3}
    The symbols $q_{-2}$ and $q_{-3}$ defined in \eqref{eq:paramm2} have respective expressions
    \begin{align}
	q_{-2} &= l_2^{-1} = (\xi\cdot\gamma_0\xi)^{-1}, \label{eq:qm2} \\
	q_{-3} &= l_2^{-3} \i\ \xi_p\xi_q\xi_j \left( [\gamma_0]_{pq} \partial_{x_i}[\gamma_0]_{ij} - 2 [\gamma_{0}]_{ij} \partial_{x_i} [\gamma_0]_{pq}\right). \label{eq:qm3}
    \end{align}
\end{lemma}

\begin{proof}[Proof of Lemma \ref{lem:qm2m3}] Using formula \eqref{eq:prodpsidos} with $(Q,P)\equiv(Q,L_0)$, and using the expansions of $\sigma_Q$ and $\sigma_{L_0}$, we get
\begin{align*}
    \sigma_{QL_0} \sim q_{-2} l_2 + (q_{-2} l_1 + q_{-3} l_2 + \frac{1}{\i} \nabla_\xi q_{-2}\cdot\nabla_x l_2) + \O (|\xi|^{-2}),
\end{align*}
In order to match the expansion $1 + 0 + \O(|\xi|^{-2})$, the expansion above must satisfy, for large $\xi$,
\begin{align*}
    q_{-2} l_2 = 1 \qandq q_{-2} l_1 + q_{-3} l_2 + \frac{1}{\i} \nabla_\xi q_{-2}\cdot\nabla_x l_2 = 0,
\end{align*}
that is, $q_{-2} = l_2^{-1} = (\xi\cdot\gamma_0\xi)^{-1}$ and
\begin{align*}
    q_{-3} = l_2^{-1} \left( -q_{-2} l_1 - \frac{1}{\i} \nabla_\xi q_{-2}\cdot\nabla_x l_2 \right) = l_2^{-3}\left( -l_2 l_1 - \i \nabla_\xi l_2\cdot\nabla_x l_2 \right).
\end{align*}
Now, we easily have $\nabla_\xi l_2 = 2\gamma_0\xi$ and $\nabla_x l_2 = \partial_{x_i} [\gamma_0]_{pq} \xi_p \xi_q \bfe_i$, where $\bfe_1,\dots,\bfe_n$ is the natural basis of $\Rm^n$. We thus deduce the expression of $q_3$
\begin{align*}
    q_{-3} = l_2^{-3} \i \left( [\gamma_0]_{pq}\xi_p\xi_q \partial_{x_i} [\gamma_0]_{ij}\xi_j - 2[\gamma_0]_{ij}\xi_j \partial_{x_i} [\gamma_0]_{pq}\xi_p\xi_q \right),
\end{align*}
from which \eqref{eq:qm3} holds. $q_{-3}$ is clearly in $S^{-3}$ from this expression, since $l_2^{-3}$ is of order $-6$. The proof is complete    
\end{proof}

We now give the expression of $M_{ij}|_{-1}$ (or rather, that of $A_0\ M_{ij}|_{-1}\ A_0$). 
\begin{proposition}[Expression of $A_0 M_{ij}|_{-1} A_0$] \label{prop:Mijm1}
    The symbol $A_0\ M_{ij}|_{-1}\ A_0$ admits the following expression for any $(i,j)$
    \begin{align}
	\begin{split}
	    A_0\ M_{ij}|_{-1}(x,\xi)\ A_0 &= \i \|\xi_0\|^{-1} \left( (\hxi_0\cdot V_j) (\Hm_i - 2\hxi_0\odot \Hm_i\hxi_0) + \hxi_0\odot \Hm_i V_j \right)\\
	    &\quad + \i\|\xi_0\|^{-1} \left( (\hxi_0\cdot V_i) (\Hm_j - 2\hxi_0\odot \Hm_j\hxi_0) + \hxi_0\odot \Hm_j V_i \right) \\
	    &\quad +\i \|\xi_0\|^{-1} \left( V_j\cdot G(x,\xi)(\hxi_0\odot V_i)+V_i\cdot G(x,\xi)(\hxi_0\odot V_j) \right),
	\end{split}
	\label{eq:Mijm1}
    \end{align}
    where we have defined $V_i := A_0\nabla u_i$, $\Hm_i := A_0\ \nabla^2 u_i\ A_0$, as well as the vector field
    \begin{align}
	G(x,\xi) := \|\xi_0\|^{2} (\i q_{-3} \xi_0 + A_0 \nabla_x q_{-2}) \in (S^0)^n. 
	\label{eq:G}
    \end{align}
\end{proposition}

\begin{proof}[Proof of Prop. \eqref{eq:Mijm1}] Assume $Q$ is a parametrix of $L_0$ modulo $\Psi^{-2}$ and consider formula \eqref{eq:dHijparam}. Since the term $\gamma:\nabla u_i\odot \nabla u_j$ is of order zero, the computation of $M_{ij}|_{-1}$ consists in computing the second term in the symbol expansion of $R_i \circ Q\circ P_j$, and the same term with $i,j$ permuted, where we denote $R_i := \gamma_0\nabla u_i\cdot\nabla$ with symbol $r_{i,1} = \i \gamma_0 \nabla u_i\cdot\xi$. Plugging $\sigma_{R_i} = r_{i,1}$, $\sigma_Q = q_{-2} + q_{-3}$ and $\sigma_{P_{i}} = p_{i,1} + p_{i,0}$ into \eqref{eq:prodpsidos2} and keeping only the terms that are homogeneous of degree $-1$ in $\xi$, we arrive at the expression
    \begin{align}
	\sigma_{R_iQP_j}|_{-1} = r_{i,1} (q_{-3}p_{j,1} + q_{-2} p_{j,0}) + \frac{1}{\i} ( p_{j,1} \nabla_\xi r_{i,1} \cdot \nabla_x q_{-2} +\nabla_\xi (q_{-2} r_{i,1}) \cdot\nabla_x p_{j,1}).
	\label{eq:RQP}
    \end{align}
    Note that the multiplications commute because the symbols of $Q$ and $R_i$ are scalar, while that of $P_j$ is matrix-valued. Since $M_{ij}|_{-1} = \sigma_{R_iQP_j}|_{-1} + \sigma_{R_jQP_i}|_{-1}$, equation \eqref{eq:Mijm1} will be proved when we show that 
    \begin{align}
	\begin{split}
	    A_0\ \sigma_{R_iQP_j}|_{-1}\ A_0  &= \i \|\xi_0\|^{-1} \Big( (\hxi_0\cdot V_i) (\Hm_j - 2\hxi_0\odot \Hm_j \hxi_0) \\
	    &\qquad + \hxi_0\odot \Hm_j V_i + V_i\cdot G(x,\xi)(\hxi_0\odot V_j) \Big).	    
	\end{split}	
	\label{eq:goal}
    \end{align}

    {\bf Proof of \eqref{eq:goal}.} Starting from \eqref{eq:RQP}, plugging the expression $r_{i,1} = \i (V_i\cdot\xi_0)$, using the identity
    \begin{align*}
	\nabla_\xi(q_{-2}r_{i,1})\cdot\nabla_x p_{j,1} = \i \xi\odot (\nabla^2 u_j \nabla_\xi(q_{-2}r_{i,1})),
    \end{align*}
    and pre- and post-multiplying by $A_0$ yields the relation
    \begin{align}
	\begin{split}
	    A_0\ \sigma_{R_iQP_j}|_{-1}\ A_0 &= \i (V_i\cdot\xi_0) (q_{-3} \i \xi_0 \odot V_j + q_{-2} \Hm_j) \\
	    &\qquad + (V_i\cdot A_0 \nabla_x q_{-2}) \i \xi_0 \odot V_j + \xi_0 \odot \Hm_j A_0^{-1} \nabla_\xi (q_{-2} r_{i,1}). 	    
	\end{split}
	\label{eq:tmp}	
    \end{align}
    Gathering the first and third terms recombines into $\i \|\xi_0\|^{-1} V_i\cdot G (\hxi_0\odot V_j)$ (the last term of \eqref{eq:goal}). On to the second and fourth terms, we first compute
    \begin{align*}
	A_0^{-1} \nabla_\xi (r_{i,1}q_{-2}) &= \i (V_i\cdot\xi_0) (-\|\xi_0\|^{-4}) 2\xi_0 + \|\xi_0\|^{-2} \i V_i = \i \|\xi_0\|^{-2} (V_i - 2(V_i\cdot\hxi_0) \hxi_0). 
    \end{align*}
    Using this calculation, the second and fourth terms in \eqref{eq:tmp} recombine into 
    \begin{align*}
	\i \|\xi_0\|^{-1} \left( (\hxi_0\cdot V_i) (\Hm_j - 2\hxi_0\odot \Hm_j\hxi_0) + \hxi_0\odot \Hm_jV_i \right) ,
    \end{align*}
    thus the argument is complete. 
\end{proof}

\subsection{Proof of Proposition \ref{prop:microloc}} \label{ssec:proofs}

{\bf Preliminaries:} By virtue of Hypothesis \ref{hyp:det}, $\nabla u_{n+1}$ may be decomposed into the basis $\nabla u_1,\dots,\nabla u_n$ by means of scalars $\mu_1,\dots, \mu_n, \mu$ such that 
\begin{align}
    \sum_{i=1}^n \frac{\mu_i}{\mu} \nabla u_i + \nabla u_{n+1} = 0.
    \label{eq:ld}
\end{align}
As seen in \cite{Monard2012b,Monard2012a}, the coefficients $\mu_1,\dots,\mu_{n+1}$ are directly computible from the power densities $\{\nabla u_i\cdot\gamma_0\nabla u_j\}_{1\le i\le j\le n+1}$ and on the other hand, we have the relation
\begin{align*}
    \frac{\mu_i}{\mu} = \frac{\det (\nabla u_1,\dots,\overbrace{\nabla u_{n+1}}^i, \dots, \nabla u_n)}{\det (\nabla u_1,\dots,\nabla u_m)}, \quad 1\le i\le n,
\end{align*}
thus $\nabla \frac{\mu_i}{\mu} = Z_i$ as defined in \eqref{eq:Zmat} for $1\le i\le n$. In the next proofs, we will use the following
\begin{lemma}\label{lem:Mmat}
    Under hypotheses \ref{hyp:det} and \ref{hyp:Z}, the following matrix-valued function 
    \begin{align}
	\Mm := \mu_i \Hm_i + \mu \Hm_{n+1}
	\label{eq:Mmat}
    \end{align}   
    is symmetric and uniformly invertible.   
\end{lemma}

\begin{proof}
    Symmetry of $\Mm$ is obvious by definition. Taking gradient of \eqref{eq:ld}, we arrive at
    \begin{align*}
	\sum_{i=1}^n Z_i \otimes \nabla u_i + \frac{\mu_i}{\mu} \nabla^2 u_i + \nabla^2 u_{n+1} = 0.	
    \end{align*}
    Pre- and post-multiplying by $A_0$, we deduce that 
    \begin{align*}
	\Mm = \mu_i \Hm_i + \mu \Hm_{n+1} = -\mu A_0 Z_i \otimes V_i = -\mu A_0 Z \Vm^T,
    \end{align*}
    where $\Vm:= [V_1|\dots|V_n]$. The proof is complete since Hyp. \ref{hyp:det} ensures that $\mu$ never vanishes and $\Vm$ is uniformly invertible, and Hyp. \ref{hyp:Z} ensures that $Z$ is uniformly invertible. 
\end{proof}

{\bf The $T_{pq}$ operators. Proof of Prop. \ref{prop:microloc}:}
As advertised in Sec. \ref{ssec:statmicroloc}, because of the algebraic form of the symbols of the linearized power density operators, it is convenient for inversion purposes to define the microlocal change of basis $T\gamma = \{T_{pq}\gamma\}_{1\le p\le q\le n}$ as in \eqref{eq:T}, i.e. 
\begin{align*}
  T_{pq}\gamma (x) := (2\pi)^{-n} \int_{\Rm^n} e^{\i\xi\cdot x} A_0^{-1}\hxi_p \odot \hxi_q A_0^{-1} : \hat\gamma(\xi)\ d\xi.
\end{align*}
To convince ourselves that this collection forms a microlocally invertible operator of $\gamma$, let us introduce the zero-th order $\Psi$DOs $P_{ijpq}$ with scalar principal symbol $\sigma_{P_{ijpq}} := (\bfe_i\cdot A_0 \hxi_p) (\bfe_j\cdot A_0 \hxi_q)$ for $1\le i,j,p,q\le n$. Then for any $1\le i\le j\le n$, the composition of operators $\sum_{p,q=1}^n P_{ijpq} \circ T_{pq}$ has principal symbol (repeated indices are summed over)
\begin{align*}
    (\bfe_i\cdot A_0 \hxi_p) (\bfe_j\cdot A_0 \hxi_q) A_0^{-1}\hxi_p \odot \hxi_q A_0^{-1} &= (\bfe_i\cdot A_0 \hxi_p) A_0^{-1}\hxi_p \odot (\bfe_j\cdot A_0 \hxi_q) A_0^{-1}\hxi_q \\
    &= \bfe_i\odot\bfe_j,
\end{align*}
where we have used the following property, true for any smooth vector field $V$:
\begin{align*}
    V = (V\cdot A_0 \hxi_p) A_0^{-1}\hxi_p.
\end{align*}
Thus for any $1\le i\le j\le n$, the composition $\sum_{p,q=1}^n P_{ijpq} \circ T_{pq}$ recovers $\gamma_{ij} = \gamma: \bfe_i\otimes\bfe_j$ up to a regularization term. This in particular justifies the estimates \eqref{eq:changebasis} and the subsequent inversion procedure. We are now ready to prove Proposition \ref{prop:microloc}.

\begin{proof}[Proof of Proposition \ref{prop:microloc}]
    From the fact that $(V_1,\dots,V_n)$ is a basis at every point and given their dotproducts $H_{ij} = V_i\cdot V_j$, we have the following formula, true for every vector field $W$:
    \begin{align}
	W = H^{pq} (W\cdot V_p) V_q.
	\label{eq:vci}
    \end{align}
    {\bf Proof of {\em (i)}: Reconstruction of the components $T_{00}\gamma$ and $\{T_{\alpha\beta}\gamma\}_{1\le\alpha\le\beta\le n-1}$.} We work with $\wtM_{ij}|_0 := A_0\ M_{ij}|_0\ A_0 = V_i\odot V_j - (\hxi_0\cdot V_i) \hxi_0 \odot V_j - (\hxi_0\cdot V_j) \hxi_0 \odot V_i$. Using \eqref{eq:vci} with $W\equiv\hxi_\alpha$, straightforward computations yield
    \begin{align*}
	\sum_{i,j,p,q} H^{qj}(\hxi_\alpha\cdot V_q) H^{pi}(\hxi_\beta\cdot V_p) \wtM_{ij}|_0 &= \hxi_\alpha\odot \hxi_\beta - (\hxi_0\cdot\hxi_\alpha) \hxi_0\odot\hxi_\beta - (\hxi_0\cdot\hxi_\beta) \hxi_0\odot\hxi_\alpha \\
	&= \left\{	\begin{array}{ccc}
	    -\hxi_0\odot\hxi_0 & \text{ if }  & \alpha= \beta = 0, \\
	    0 &\text{ if }  & 0 = \alpha \ne \beta, \\
	    \hxi_\alpha\odot \hxi_\beta &\text{ if }  & \alpha\ne 0, \beta\ne 0.
	\end{array}
	\right.
    \end{align*}
    which means that upon defining $Q_{\alpha\beta ij}\in \Psi^0$ with scalar principal symbols 
    \begin{align*}
	\sigma_{Q_{00 ij}} &:= -\sum_{p,q} H^{qj}(\hxi_0\cdot V_q) H^{pi}(\hxi_0\cdot V_p), \\
	\sigma_{Q_{\alpha\beta ij}} &:= \sum_{p,q} H^{qj}(\hxi_\alpha\cdot V_q) H^{pi}(\hxi_\beta\cdot V_p), \quad 1\le \alpha\le \beta\le n-1,
    \end{align*}
    relation \eqref{eq:param1} is satisfied in the sense of operators since the previous calculation amounts to computing the principal symbol of the composition of operators in \eqref{eq:param1}. \\ \smallskip
    {\bf Proof of {\em (ii)}: Reconstruction of the components $\left\{ T_{0\alpha}\gamma \right\}_{1\le\alpha\le n-1}$.} It remains to construct appropriate operators that will map $\dH (\gamma)$ to the components $T_{0\alpha}\gamma$ for $1\le \alpha\le n-1$, which is where the additional measurements $\dH_{i,n+1}$ come into play. Let $(\mu_1,\dots,\mu_n,\mu)$ as in \eqref{eq:ld} and construct the $\Psi$DO $\{L_i(\gamma)\}_{i=1}^n$ as in \eqref{eq:Li}. It is easy to see that, since the $\mu_i$ are only functions of $x$, the terms of fixed homogeneity in the symbol expansion of $L_i$ satisfy
    \begin{align*}
	\sigma_{L_i}|_{k} = \mu_j M_{ij}|_{k} + \mu M_{i,n+1}|_{k}, \qquad k = 0,-1,-2,\dots.
    \end{align*}
    Then from equation \eqref{eq:Mij0_2} and relation \eqref{eq:ld}, we deduce that $\sigma_{L_i}|_0 = 0$, so that $L_i\in \Psi^{-1}$. Moreover, using equation \eqref{eq:Mijm1} together with relation \eqref{eq:ld}, we deduce that
    \begin{align*}
	\tilde\sigma_{L_i}|_{-1} = A_0\ \sigma_{L_i}|_{-1}\ A_0 = \i \|\xi_0\|^{-1} \left( (\hxi_0\cdot V_i) (\Mm - 2\hxi_0\odot \Mm\hxi_0) + \hxi_0 \odot \Mm V_i  \right)
    \end{align*}
    is now the principal symbol of $L_i$. Using relation \eqref{eq:vci} with $W\equiv \Mm^{-1} \hxi_\alpha$, the symmetry of $\Mm$ and multiplying by $\Mm$, we have the relation 
    \begin{align*}
	\hxi_\alpha = H^{pq} (\hxi_\alpha\cdot\Mm^{-1}V_p) \Mm V_q.
    \end{align*}
    Using this relation, we deduce the following calculation, for $1\le \alpha\le n-1$
    \begin{align}
	H^{pi} (\hxi_\alpha\cdot \Mm^{-1} V_p)\ \tilde\sigma_{L_i}|_{-1} = \i \|\xi_0\|^{-1} \left( (\hxi_0\cdot \Mm^{-1}\hxi_\alpha) (\Mm - 2\hxi_0\odot \Mm\hxi_0) + \hxi_0 \odot \hxi_\alpha \right).
	\label{eq:manip}
    \end{align} 
    While the second term gives us the missing components $T_{0\alpha}\gamma$, we claim that the first one is spanned by $\hxi_0\odot\hxi_0$ and $\{\hxi_\alpha\odot\hxi_\beta\}_{1\le\alpha\le\beta\le n-1}$. Indeed we have
    \begin{align*}
	(\Mm - 2\hxi_0\odot \Mm\hxi_0):(\hxi_0\odot\hxi_\alpha) &=  0, \quad 1\le \alpha\le n-1, \\
	(\Mm - 2\hxi_0\odot \Mm\hxi_0):(\hxi_0\odot\hxi_0) &= - \hxi_0\cdot\Mm\hxi_0, \\
	(\Mm - 2\hxi_0\odot \Mm\hxi_0):(\hxi_\alpha\odot\hxi_\beta) &= \hxi_\alpha\cdot\Mm\hxi_\beta, \quad 1\le\alpha\le\beta\le n-1,
    \end{align*}
    so we deduce that 
    \begin{align}
	\Mm - 2\hxi_0\odot \Mm\hxi_0 = - (\hxi_0\cdot\Mm\hxi_0)\ \hxi_0\odot\hxi_0 + \sum_{1\le\alpha,\beta\le n-1} (\hxi_\alpha\cdot\Mm\hxi_\beta)\ \hxi_\alpha\odot\hxi_\beta.
	\label{eq:Rexpr}
    \end{align}

    In light of these algebraic calculations, we now build the parametrices. Let $L_0^{\frac{1}{2}}\in \Psi^{1}$, $B_{\alpha i}\in \Psi^0$, $R\in (\Psi^0)^{n\times n}$, $R_\alpha\in \Psi^0$ and $R_{\alpha\beta} \in \Psi^0$ the $\Psi$DOs with respective principal symbols 
    \begin{align*}
	\sigma_{L_0^{\frac{1}{2}}} &= -\i \|\xi_0\|, \quad \sigma_{B_{\alpha i}} = H^{pi} (\hxi_\alpha\cdot \Mm^{-1} V_p), \quad \sigma_R = \Mm - 2\hxi_0\odot\Mm\hxi_0, \\
	\sigma_{R_\alpha} &= \hxi_0\cdot\Mm^{-1} \hxi_\alpha, \quad \sigma_{R_{\alpha\beta}} = \hxi_\alpha\cdot\Mm\hxi_\beta.
    \end{align*}
    Then the relation \eqref{eq:manip} implies \eqref{eq:param2} at the principal symbol level. The operator $R$ can indeed be expressed as the following zero-th order linear combination of the components $T_{00}$ and $\{T_{\alpha\beta}\}_{1\le \alpha,\beta\le n-1}$: 
    \begin{align*}
	R &= - R_{00} T_{00} + \sum_{1\le \alpha, \beta\le n-1} R_{\alpha\beta} T_{\alpha\beta} \\
	&= \sum_{i,j=1}^n \left( -R_{00} Q_{00ij} + \sum_{1\le\alpha,\beta\le n-1} R_{\alpha\beta} Q_{\alpha\beta ij} \right)\circ dH_{ij} \mod \Psi^{-1}, 
    \end{align*}
    so that the left-hand side of \eqref{eq:param2} is expressed as a post-processing of measurement operators $\dH_{ij}$ only. The proof is complete.  
\end{proof}

\section{Explicit inversion} \label{sec:explicit}

\subsection{Preliminaries and notation} \label{ssec:prelim2}

For a matrix $A$ with columns $A_1,\dots,A_n$ and $(\bfe_1,\dots,\bfe_n)$ the canonical basis, one has the following representation
\begin{align*}
    A = \sum_{j=1}^n A_j\otimes \bfe_j \qandq A^T = \sum_{j=1}^n \bfe_j\otimes A_j.
\end{align*}
More generally, for two matrices $A = [A_1|\dots|A_n]$ and $B = [B_1|\dots|B_n]$, we have the relation
\begin{align*}
    \sum_{j=1}^n A_j\otimes B_j = AB^T.
\end{align*}
Finally, for $A$ a matrix and $V = [V_1|\dots|V_n]$, the sum $A_{ij}V_j$ is nothing but the $i$-th column of the matrix $VA^T$.

\subsection{Derivation of \eqref{eq:gammaelim} from Hypothesis \ref{hyp:det}:} \label{ssec:der1}

Let us start from $n$ solutions $(u_1,\dots,u_n)$ fulfilling Hypothesis \ref{hyp:det}, and let $(v_1,\dots,v_n)$ the corresponding solutions of \eqref{eq:conductivity1}. We also denote $[\nabla U] := [\nabla u_1|\dots|\nabla u_n]$ and $[\nabla V]$ similarly. We first mention that for any vector field $V$, we have the following formulas
\begin{align}
    V = H^{pq} (V\cdot\gamma_0\nabla u_p) \nabla u_q = H^{pq} (V\cdot\nabla u_p) \gamma_0\nabla u_q,
    \label{eq:idV}
\end{align}
which also amounts to the following matrix relations
\begin{align}
    H^{pq} (\nabla u_p\otimes \nabla u_q) \gamma_0 = H^{pq} \gamma_0 (\nabla u_p\otimes\nabla u_q) = \Imm_n.
    \label{eq:idM}
\end{align}

From the relation
\begin{align*}
    \dH_{ij} = (\gamma\nabla u_i + \gamma_0\nabla v_i)\cdot\nabla u_j + \gamma_0 \nabla v_j\cdot\nabla u_i, \quad 1\le i,j\le n,
\end{align*}
we deduce, using \eqref{eq:idV},
\begin{align}
    \gamma\nabla u_i + \gamma_0\nabla v_i = H^{pq} \left( \dH_{ip} - \gamma_0\nabla v_p\cdot\nabla u_i  \right) \gamma_0\nabla u_q, \quad 1\le i\le n.
    \label{eq:rel1}
\end{align}
The previous equation allows us express $\gamma$ in terms of the remaining unknowns $(v_1,\dots,v_n)$. Indeed, taking the tensor product of \eqref{eq:rel1} with $H^{ij} \gamma_0\nabla u_j$ and summing over $i$ yields
\begin{align*}
    \gamma + \gamma_0 \nabla v_i\otimes \nabla u_j \gamma_0 H^{ij} &= H^{pq} \left( \dH_{ip} - \gamma_0\nabla v_p\cdot\nabla u_i  \right) (\gamma_0\nabla u_q \otimes \nabla u_j \gamma_0 H^{ij}) \\
    &= \dH_{ip} \gamma_0 (H^{pq}\nabla u_q\otimes H^{ij}\nabla u_j )\gamma_0 - \gamma_0 \nabla u_q\otimes \nabla v_p \gamma_0 H^{pq},
\end{align*}
where we have used the identity \eqref{eq:idV} in the last right-hand side. We may rewrite this as
\begin{align}
    \gamma = \gamma_0 \left( \dH_{ip} (H^{pq}\nabla u_q\otimes H^{ij}\nabla u_j) - 2 H^{ij} \nabla v_i\odot \nabla u_j \right) \gamma_0.
    \label{eq:reconsgamma}
\end{align}
One may notice that the above expression is indeed a symmetric matrix. In matrix notation, using the preliminaries, we arrive at the expression \eqref{eq:gammaelim}.

\subsection{Algebraic equations obtained by considering additional solutions:} \label{ssec:der2}

Let us now add another solution $u_{n+1}$ with corresponding solution $v_{n+1}$ at order $\O(\varepsilon)$. By virtue of Hypothesis \ref{hyp:det}, as in section \ref{ssec:proofs}, $\nabla u_{n+1}$ may be expressed in the basis $(\nabla u_1,\dots,\nabla u_n)$ as
\begin{align}
    \sum_{i=1}^n \frac{\mu_i}{\mu} \nabla u_i + \nabla u_{n+1} = 0,
    \label{eq:lindep}
\end{align}
where the coefficients $\mu_i$ can be expressed as ratios of determinants, or equivalently, computable from the power densities at order $\varepsilon^0$, see \cite[Appendix A.3]{Monard2012b}. For $1\le i\le n$, we define $Z_i := \nabla (\mu^{-1} \mu_i)$, and notice that we have the following two algebraic relations
\begin{align}
    \sum_{i=1}^n Z_i \cdot\gamma_0\nabla u_i = 0 \qandq \sum_{i=1}^n Z_i^\flat \wedge du_i = 0.
    \label{eq:algrelu}
\end{align}
The first one is obtained obtained after applying the operator $\nabla\cdot(\gamma_0\cdot)$ to \eqref{eq:lindep} and the second one is obtained after applying an exterior derivative to \eqref{eq:lindep} . 

Moving on to the study of the corresponding $v_{n+1}$ solution, we write 
\begin{align*}
    \dH_{n+1,j} + \frac{\mu_i}{\mu} \dH_{ij} = \left( \nabla v_{n+1} + \frac{\mu_i}{\mu}\nabla v_i \right) \cdot\gamma_0\nabla u_j, \quad 1\le j\le n,
\end{align*}
where we have cancelled sums of the form \eqref{eq:lindep}. Using the identity \eqref{eq:idV}, we deduce that
\begin{align}
    \nabla v_{n+1} + (\mu^{-1}\mu_i) \nabla v_i = H^{pq} \left(\dH_{n+1,p} + (\mu^{-1}\mu_i) \dH_{ip}\right) \nabla u_q.
    \label{eq:vvi}
\end{align}
Taking exterior derivative of the previous relation yields
\begin{align}
    Z_i^\flat\wedge dv_i = d\left( H^{pq} (\dH_{n+1,p} + (\mu^{-1}\mu_i) \dH_{ip}) \right) \wedge du_q.
    \label{eq:algrelv}
\end{align}
We now apply $\nabla\cdot(\gamma_0\cdot)$ to \eqref{eq:vvi}, the left-hand side becomes
\begin{align*}
    \nabla\cdot(\gamma_0 (\nabla v_{n+1} &+ (\mu^{-1}\mu_i)\nabla v_i) ) \dots \\
    &= \nabla\cdot(\gamma_0\nabla v_{n+1}) + Z_i\cdot\gamma_0\nabla v_i + (\mu^{-1}\mu_i)\nabla\cdot(\gamma_0\nabla v_i) \\
    &= - \nabla\cdot(\gamma \nabla u_{n+1}) + Z_i\cdot\gamma_0\nabla v_i - (\mu^{-1}\mu_i) \nabla\cdot(\gamma\nabla u_i) \\
    &= Z_i\cdot\gamma_0\nabla v_i - \nabla\cdot(\gamma (\nabla u_{n+1} + (\mu^{-1} \mu_i)\nabla u_i)) + Z_i \cdot\gamma\nabla u_i \\
    &= Z_i\cdot(\gamma_0\nabla v_i+\gamma\nabla u_i),
\end{align*}
thus we arrive at the equation
\begin{align*}
    Z_i\cdot(\gamma_0\nabla v_i+\gamma\nabla u_i) = \nabla \left( H^{pq} (\dH_{n+1,p} + (\mu^{-1}\mu_i) \dH_{ip}) \right) \cdot\gamma_0 \nabla u_q =: Y_q\cdot\gamma_0\nabla u_q,
\end{align*}
where the vector fields
\begin{align}
    Y_q := \nabla \left( H^{pq} (\dH_{n+1,p} + (\mu^{-1}\mu_i) \dH_{ip}) \right),\quad 1\le q\le n,
    \label{eq:Yq}
\end{align}
are known from the data $\dH$. Combining the latter equation with \eqref{eq:rel1}, we obtain
\begin{align*}
    (Z_i\cdot \gamma_0\nabla u_q) H^{pq} \left( \dH_{ip} - \gamma_0\nabla v_p\cdot\nabla u_i  \right) = Y_q \cdot\gamma_0 \nabla u_q,
\end{align*}
which we recast as
\begin{align*}
    (Z_i\cdot \gamma_0\nabla u_q) H^{pq}  (\gamma_0\nabla v_p\cdot\nabla u_i ) = (Z_i\cdot \gamma_0\nabla u_q) H^{pq} \dH_{ip} - Y_q \cdot\gamma_0 \nabla u_q.
\end{align*}
The left-hand side can be considerably simplified by noticing that the second equation of \eqref{eq:algrelu} implies $[\nabla U]Z^T = Z[\nabla U]^T$. With this fact in mind, the left-hand side looks like $X_p\cdot\nabla v_p$, where we compute
\begin{align*}
    X_p = H^{pq} \gamma_0 \nabla u_i \otimes Z_i \gamma_0 \nabla u_q &= \gamma_0 [\nabla U]Z^T \gamma_0 [\nabla U] H^{-1} \bfe_p \\
    &= \gamma_0 Z [\nabla U]^T [\nabla U]^{-T} \bfe_p = \gamma_0 Z_p.
\end{align*}

Finally, we obtain the more compact equation
\begin{align}
    \sum_{p=1}^n \gamma_0 Z_p\cdot\nabla v_p = f, \where\quad f := (H^{pq} \dH_{ip}\ Z_i - Y_q)\cdot \gamma_0\nabla u_q,
    \label{eq:algrelv2}
\end{align}
with $Y_q$ given in \eqref{eq:Yq}.

\begin{remark}[On algebraic inversion]
    In equations \eqref{eq:algrelv} and \eqref{eq:algrelv2}, the only unknown is the matrix $[\nabla V]:= [\nabla v_1,\dots,\nabla v_n]$. Equations \eqref{eq:algrelv} and \eqref{eq:algrelv2} give us the projection of that matrix onto the space $Z A_n(\Rm)$ and onto the line $\Rm \gamma_0 Z$ respectively. As in the non-linear case \cite{Monard2012b,Monard2012a}, we expect that a rich enough set of such equations provided by a certain number of additional solutions $(u_{n+1}, \dots, u_{n+l})$ leads to a pointwise, algebraic reconstruction of $[\nabla V]$, however we do not follow that route here.
\end{remark}

\subsection{Proof of Proposition \ref{prop:scesv} and Theorem \ref{thm:explicit}} \label{ssec:proofs2}
We now show that provided that we use {\em one} additional solution $u_{n+1}$ (on top of the basis $(u_1,\dots,u_n)$) such that the matrix $Z$ is of full rank, then we can reconstruct $(v_1,\dots,v_n)$ via a strongly coupled elliptic system of the form \eqref{eq:scesv}, after which we can reconstruct $\gamma$ from $(\nabla v_1,\dots,\nabla v_n)$ by formula \eqref{eq:gammaelim}. We now show how to derive this elliptic system.

\begin{proof}[Proof of Proposition \ref{prop:scesv}]
    According to Hypothesis \ref{hyp:Z}, the matrix $Z=[Z_1|\dots|Z_n]$ has full rank and we recall the important equations
    \begin{align}
	\sum_{p=1}^n \gamma_0 Z_p\cdot\nabla v_p &= f\qandq \sum_{i=1}^n Z_i^\flat \wedge dv_i = \omega, \where \label{eq:eqsv} \\
	\omega = Y_q^\flat \wedge du_q, &\qquad Y_q:= \nabla (H^{pq}(\dH_{n+1,p} + (\mu^{-1}\mu_i) \dH_{ip})), \label{eq:omY}
    \end{align}
    and where $f$ is given in \eqref{eq:algrelv2}. Assuming that $Z$ has full rank, the family $(Z_1,\dots,Z_n)$ is a frame with dotproducts defined as $\Xi_{ij} = Z_i\cdot Z_j$, and in this case we define its dual frame $Z_i^\star := \Xi^{ij}Z_j$ for $1\le i\le n$, such that $Z_i^\star\cdot Z_j = \delta_{ij}$, i.e. with $Z^\star$ the matrix with columns $Z^\star_j$, we have the relation $Z^{\star} = Z^{-T}$. The second equation of \eqref{eq:eqsv} may be rewritten as
    \begin{align}
	Z^\star_q\cdot\nabla v_p - Z^\star_p\cdot\nabla v_q = \omega(Z_p^\star,Z_q^\star), \quad 1\le p,q\le n.
	\label{eq:Zstar}
    \end{align}
    Applying the differential operator $Z_i^\star\cdot\nabla$ to the first equation of \eqref{eq:eqsv}, we obtain
    \begin{align}
	\sum_{p=1}^n (Z_i^\star\cdot\nabla) (\gamma_0 Z_p\cdot\nabla) v_p = (Z_i^\star\cdot\nabla) f.
	\label{eq:tmpSCES}
    \end{align}
    Using \eqref{eq:Zstar}, we may rewrite the left-hand side of \eqref{eq:tmpSCES} as
    \begin{align*}
	(Z_i^\star\cdot\nabla) (\gamma_0 Z_p\cdot\nabla) v_p &= [Z_i^\star, \gamma_0 Z_p]\cdot\nabla v_p + (\gamma_0 Z_p\cdot\nabla) (Z_i^\star\cdot\nabla)  v_p \\
	&= [Z_i^\star, \gamma_0 Z_p]\cdot\nabla v_p + (\gamma_0 Z_p\cdot\nabla) (Z_p^\star\cdot\nabla) v_i \dots\\
	&\qquad+ (\gamma_0 Z_p\cdot\nabla) ( \omega (Z_p^\star, Z_i^\star) ),
    \end{align*}
    where we have introduced the Lie bracket of two vector fields, which may be written in the Euclidean connection
    \begin{align}
	[X,Y] := (X\cdot\nabla) Y - (Y\cdot\nabla) X.
	\label{eq:lie}
    \end{align}
    Plugging the last calculation into \eqref{eq:tmpSCES} (repeated indices are summed over)
    \begin{align}
	(\gamma_0 Z_p\cdot\nabla) (Z_p^\star\cdot\nabla) v_i + [Z_i^\star, \gamma_0 Z_p]\cdot\nabla v_p = (Z_i^\star\cdot\nabla) f - (\gamma_0 Z_p\cdot\nabla)( \omega (Z_p^\star, Z_i^\star) ).
	\label{eq:tmpSCES2}
    \end{align}
    We now look more closely at the principal part of this equation. The first term may be written as 
    \begin{align*}
	\gamma_0 Z_p\otimes Z_p^\star :\nabla^2 v_i + ((\gamma_0 Z_p\cdot\nabla)Z_p^\star)\cdot\nabla v_i = \gamma_0:\nabla^2 v_i + ((\gamma_0 Z_p\cdot\nabla)Z_p^\star)\cdot\nabla v_i,
    \end{align*}
    where we have used that $Z_p\otimes Z_p^\star = \Imm_n$. We thus obtain a strongly coupled elliptic system of the form \eqref{eq:scesv}, where
    \begin{align}
	W_{ij} &:=  (\nabla\cdot\gamma_0 - ((\gamma_0 Z_p\cdot\nabla) Z_p^\star))\ \delta_{ij} - [Z_i^\star, \gamma_0 Z_j], \quad 1\le i,j\le n, \label{eq:Wij} \\
	f_i    &:= - Z_i^\star\cdot\nabla f + (\gamma_0 Z_p\cdot\nabla) ( \omega (Z_p^\star,Z_i^\star) ),\quad 1\le i\le n. \label{eq:fi}	
    \end{align}
    This concludes the proof.
\end{proof}

In order to assess the properties of system \eqref{eq:scesv}, we recast it as an integral equation as follows: Let us call $L_0 := -\nabla\cdot (\gamma_0 \nabla )$, and define $L_0^{-1}: H^{-1}(X) \ni f\mapsto u\in H_0^1(X)$, where $u$ is the unique solution to the equation
\begin{align}
    -\nabla\cdot(\gamma_0\nabla u) = f \quad (X), \quad u|_{\partial X} = 0.
    \label{eq:Lzero}
\end{align}
By the Lax-Milgram theorem (see e.g. \cite{evans}), one can establish that such solutions satisfy an estimate of the form $\|u\|_{H^1_0(X)}\le C \|f\|_{H^{-1}(X)}$, where $C$ only depends on $X$ and the constant of ellipticity of $\gamma_0$, thus $L_0^{-1}:H^{-1} (X)\to H_0^1(X)$ is continuous, and by Rellich imbedding (i.e. the fact that the injection $L^2\to H^{-1}$ is compact), $L_0^{-1}: L^2(X) \to H_0^1(X)$ is {compact}.

Applying the operator $L_0^{-1}$ to \eqref{eq:scesv}, we arrive at the integral system
\begin{align}
    v_i + \sum_{j=1}^n L_0^{-1} (W_{ij}\cdot\nabla v_j) = h_i := L_0^{-1} f_i\quad (X), \quad 1\le i\le n,
    \label{eq:SCES3}
\end{align}
where it is easy to establish that for $1\le i,j\le n$, the operator
\begin{align}
    P_{ij}: H_0^1(X) \ni v \to P_{ij} v := L_0^{-1} (W_{ij}\cdot\nabla v) \in H_0^1(X)
    \label{eq:Pij}
\end{align}
is compact whenever the vector fields $W_{ij}$ are bounded. In vector notation, if we define the vector space $\H = (H_0^1 (X))^n$, $\bfv = (v_1,\dots,v_n)$, $\bfh = (h_1,\dots,h_n)$ and for $\bfv\in\H$,
\begin{align}
    \bfP\bfv := (P_{1j}v_j, P_{2j}v_j,  \dots, P_{nj}v_j) \in \H,
    \label{eq:P}
\end{align}
we have that $\bfP:\H\to\H$ is a compact linear operator, and the system \eqref{eq:scesv} is reduced to the following Fredholm (integral) equation
\begin{align}
    (\bfI + \bfP)\bfv = \bfh.
    \label{eq:scesint}
\end{align}

Note here that the operator $\bfP$ defined in \eqref{eq:P} depends only on $\gamma_0$ and the solutions $u_i$, so that the injectivity properties depend on the $\gamma_0$ around which we pose the problem, in particular, whether one can fulfill hypotheses \ref{hyp:det} and \ref{hyp:Z}. 

\paragraph{Injectivity and stability.}

Equation \eqref{eq:scesint} satisfies a Fredholm alternative. In particular, if $-1$ is not an eigenvalue of $\bfP$, \eqref{eq:scesint} admits a unique solution $\bfv\in\H$ (injectivity), $(\bfI+\bfP)^{-1}:\H\to\H$ is well-defined and continuous and $\bfv$ satisfies the estimate
\begin{align}
    \|\bfv\|_{\H} \le \|(\bfI+\bfP)^{-1}\|_{\L(\H)} \|\bfh\|_{\H},
    \label{eq:stabvh}
\end{align}
from which we deduce stability below. In the statement of Theorem \ref{thm:explicit}, the fact that ``system \eqref{eq:scesv} with trivial right-hand sides admits no non-trivial solution'' precisely means that $-1$ is not an eigenvalue of the operator $\bfP$.

\begin{remark}[Injectivity when $\gamma_0$ is constant]
    When $\gamma_0$ is constant, constructing $(u_1,\dots,u_{n+1})$ as in Remark \ref{rem:const} yields $Z=Q$ a constant matrix. In particular, the commutators $[Z_i^\star, \gamma_0 Z_j]$ vanish in the expression \eqref{eq:Wij} of $W_{ij}$. Thus system \eqref{eq:scesv} is decoupled and clearly injective. By continuity, we also obtain that \eqref{eq:scesv} is injective for $\gamma_0$ (not necessarily scalar) sufficiently close to a constant.
\end{remark}

We now prove Theorem \ref{thm:explicit}.
\begin{proof}[Proof of Theorem \ref{thm:explicit}]
    Starting from the integral version \eqref{eq:scesint} of the elliptic system \eqref{eq:scesv} in the case where $-1\notin \sp(\bfP)$, then the Fredholm alternative implies \eqref{eq:stabvh}. In order to translate inequality \eqref{eq:stabvh} into a stability statement, we must bound $\bfh$ in terms of the measurements $\{\dH_{ij}\}$. We have for $1\le i\le n$,
    \begin{align*}
	\|h_i\|_{H_0^1(X)} \le \|L_0^{-1}\|_{\L(H^{-1}, H_0^1)} \|f_i\|_{H^{-1}(X)},
    \end{align*}
    and since $f_i$, expressed in \eqref{eq:fi} involves the $\dH_{ij}$ and their derivatives up to second order, if we assume all other multiplicative coefficients to be uniformly bounded, we obtain an estimate of the form
    \begin{align*}
	\|h_i\|_{H_0^1(X)} \le C \|\dH\|_{H^1(X)}, \where\quad \|\dH\|_{H^1(X)} := \!\!\!\! \sum_{1\le i\le n,\ i\le j\le n+1} \!\!\!\! \!\!\!\! \|\dH_{ij}\|_{H^1(X)},
    \end{align*}
    thus we obtain in the end, an estimate of the form
    \begin{align}
	\|\bfv\|_{H^1_0(X)} \le C \|\dH\|_{H^1(X)}.
	\label{eq:stabv}
    \end{align}
    Once $\bfv$ is reconstructed, we can reconstruct $\gamma$ uniquely from $\dH$ and $[\nabla V]$ using formula \eqref{eq:gammaelim}, with the stability estimate
    \begin{align}
	\|\gamma\|_{L^2(X)} \le C \|\dH\|_{H^1(X)}.
	\label{eq:stabgamma}
    \end{align}
    \smallskip
    {\bf Regaining one derivative back on $\tr (\gamma_0^{-1}\gamma)$:} In order to see that $\tr (\gamma_0^{-1}\gamma)$ satisfies a gradient equation that improves the stability of its reconstruction, the quickest way is to linearize \cite[Equation (7)]{Monard2012a} derived in the non-linear case, which reads as follows:
    \begin{align*}
	\nabla\log\det \gamma^\varepsilon = \nabla\log\det H^\varepsilon + 2 \left( (\nabla (H^\varepsilon)^{jl})\cdot \gamma^\varepsilon \nabla u^\varepsilon_l \right) \nabla u^\varepsilon_j,
    \end{align*}
    where $H^\varepsilon$ is the $n\times n$ matrix of power densities $H^\varepsilon_{ij} = \nabla u_i^\varepsilon\cdot\gamma^\varepsilon\nabla u_j^\varepsilon$ and $(H^\varepsilon)^{jl}$ is the $(j,l)$-th entry of $(H^\varepsilon)^{-1}$. Plugging the expansions $\gamma^\varepsilon = \gamma_0 + \varepsilon \gamma$, $u_i^\varepsilon = u_i + \varepsilon v_i$, $H^\varepsilon_{ij} = H_{ij} + \varepsilon\dH_{ij}$, and using the fact that 
    \begin{align*}
	(H^\varepsilon)^{jl} = H^{jl} - \varepsilon (H^{-1} \dH H^{-1} )^{jl} + \O (\varepsilon^2),
    \end{align*}
    the linearized equation at $\O(\varepsilon)$ reads
    \begin{align*}
	\frac{1}{2} \nabla \tr(\gamma_0^{-1}\gamma) &= \frac{1}{2} \nabla \tr (H^{-1}\dH) + (\nabla H^{jl} \cdot\gamma_0 \nabla u_l) \nabla v_j + (\nabla H^{jl} \cdot\gamma_0 \nabla v_l) \nabla u_j \\
	&\quad + (\nabla H^{jl}\cdot\gamma\nabla u_l) \nabla u_j - \left( \nabla (H^{-1}\dH H^{-1})^{jl} \cdot\gamma_0 \nabla u_l \right) \nabla u_j.
    \end{align*}
    From this equation, and using the stability estimates \eqref{eq:stabv} and \eqref{eq:stabgamma}, it is straighforward to establish the estimate
    \begin{align*}
	\| \tr(\gamma_0^{-1}\gamma) \|_{H^1(X)} \le C \|\dH \|_{H^1(X)},
    \end{align*}
    and thus the proof is complete. 
\end{proof}



\medskip
Received xxxx 20xx; revised xxxx 20xx.
\medskip


\begin{thebibliography}{10}

\bibitem{ABCTF-SIAP-08}
{\sc H.~Ammari, E.~Bonnetier, Y.~Capdeboscq, M.~Tanter, and M.~Fink}, {\em
  Electrical impedance tomography by elastic deformation}, SIAM J. Appl. Math.,
  68 (2008), pp.~1557--1573.

\bibitem{AS-IP-12}
{\sc S.~R. Arridge and O.~Scherzer}, {\em Imaging from coupled physics},
  Inverse Problems, 28 (2012), p.~080201.

\bibitem{B-Irvine-12}
{\sc G.~Bal}, {\em {Hybrid Inverse Problems and Systems of Partial Differential
  Equations}}, {arXiv:1210.0265}.

\bibitem{B-UMEIT-12}
\leavevmode\vrule height 2pt depth -1.6pt width 23pt, {\em {Cauchy problem for
  Ultrasound modulated EIT}}, To appear in Anal. PDE. {arXiv:1201.0972v1},
  (2013).

\bibitem{B-IO-12}
\leavevmode\vrule height 2pt depth -1.6pt width 23pt, {\em {Hybrid inverse
  problems and internal functionals}}, Inside Out, Cambridge University Press,
  Cambridge, UK, G. Uhlmann, Editor, 2012.

\bibitem{Bal2011a}
{\sc G.~Bal, E.~Bonnetier, F.~Monard, and F.~Triki}, {\em Inverse diffusion
  from knowledge of power densities}, to appear in Inverse Problems and
  Imaging,  (2013).
\newblock {arXiv:1110.4577}.

\bibitem{BNSS-JIIP-13}
{\sc G.~Bal, W.~Naetar, O.~Scherzer, and J.~Schotland}, {\em Numerical
  inversion of the power density operator}, to appear in J. Ill-posed Inverse
  Problems,  (2013).

\bibitem{BU-CPAM-12}
{\sc G.~Bal and G.~Uhlmann}, {\em Reconstruction of coefficients in scalar
  second-order elliptic equations from knowledge of their solutions}, to appear in C.P.A.M., (2013)
  {arXiv:1111.5051}.

\bibitem{Capdeboscq2009}
{\sc Y.~Capdeboscq, J.~Fehrenbach, F.~de~Gournay, and O.~Kavian}, {\em Imaging
  by modification: Numerical reconstruction of local conductivities from
  corresponding power density measurements}, SIAM Journal on Imaging Sciences,
  2 (2009), pp.~1003--1030.

\bibitem{evans}
{\sc L.~Evans}, {\em Partial Differential Equations}, Graduate Studies in
  Mathematics Vol.19, AMS, 1998.

\bibitem{Folland1995}
{\sc G.~B. Folland}, {\em Introduction to Partial Differential Equations},
  Princeton University Press, Princeton New Jersey, 1995.

\bibitem{GS-SIAP-09}
{\sc B.~Gebauer and O.~Scherzer}, {\em Impedance-acoustic tomography}, SIAM J.
  Applied Math., 69(2) (2009), pp.~565--576.

\bibitem{GS-CUP-94}
{\sc A.~Grigis and j.~Sj{\"o}strand}, {\em Microlocal Analysis for Differential
  Operators: An Introduction}, Cambridge University Press, 1994.

\bibitem{Kuchment2012}
{\sc P.~Kuchment}, {\em Mathematics of hybrid imaging. a brief review.}, to
  appear in The Mathematical Legacy of Leon Ehrenpreis,  (2012).

\bibitem{Kuchment2011a}
{\sc P.~Kuchment and L.~Kunyansky}, {\em 2d and 3d reconstructions in
  acousto-electric tomography}, Inverse Problems, 27 (2011).

\bibitem{Kuchment2011}
{\sc P.~Kuchment and D.~Steinhauer}, {\em Stabilizing inverse problems by
  internal data}, Inverse Problems, 28 (2012), p.~4007.
\newblock arXiv:1110.1819.

\bibitem{Monard2012b}
{\sc F.~Monard}, {\em {T}aming unstable inverse problems. {M}athematical routes
  toward high-resolution medical imaging modalities}, PhD thesis, Columbia
  University, 2012.

\bibitem{Monard2012a}
{\sc F.~Monard and G.~Bal}, {\em Inverse anisotropic conductivity from power
  densities in dimension $n \ge 3$}, to appear in CPDE,  (2013).

\bibitem{Monard2011}
\leavevmode\vrule height 2pt depth -1.6pt width 23pt, {\em Inverse anisotropic diffusion from power
  density measurements in two dimensions}, Inverse Problems, 28 (2012),
  p.~084001.
\newblock arXiv:1110.4606.

\bibitem{Monard2011a}
\leavevmode\vrule height 2pt depth -1.6pt width 23pt,  {\em Inverse diffusion problems with redundant
  internal information}, Inv. Probl. Imaging, 6 (2012), pp.~289--313.
\newblock arXiv:1106.4277.

\bibitem{S-SP-2011}
{\sc O.~Scherzer}, {\em Handbook of Mathematical Methods in Imaging}, Springer
  Verlag, New York, 2011.

\bibitem{SU-IO-12}
{\sc P.~Stefanov and G.~Uhlmann}, {\em Multi-wave methods by ultrasounds},
  Inside out, Cambridge University Press (G. Uhlmann, Ed.),  (2012).


\end{thebibliography}
\end{document}